\documentclass[10pt, a4paper]{article}
\usepackage[numbers]{natbib}
\usepackage{amsmath,amssymb,amsfonts,setspace}
\usepackage[amsthm,thmmarks]{ntheorem}
\usepackage{mathrsfs}
\usepackage[colorlinks=true]{hyperref}
\usepackage{xy}

{\theoremstyle{change}
\newtheorem{proposition}{Proposition}[section]
\newtheorem{lemma}[proposition]{Lemma}
\newtheorem{theorem}[proposition]{Theorem}
\newtheorem{corollary}[proposition]{Corollary}
}

\newcommand{\deph}{\textbf}

\newcommand{\imag}{\text{\upshape i}}
\renewcommand{\phi}{\varphi}
\renewcommand{\epsilon}{\varepsilon}
\renewcommand{\theta}{\vartheta}

\newcommand{\lh}{{\ell_{\mathrm{H}}}}
\newcommand{\lc}{{\ell_{\mathrm{c}}}}
\newcommand{\lj}{{\ell_{\mathrm{J}}}}
\newcommand{\lr}{{\ell_{\mathrm{r}}}}
\newcommand{\lt}{\lambda}

\newcommand{\set}[2]{\left\{ #1 \,\middle|\, #2 \right\}}
\newcommand{\alev}[1]{\,\,\left[#1\right]}

\newcommand{\bs}{\boldsymbol}

\newcommand{\operp}{\mathbin{\scriptstyle \perp\mkern-15.7mu\bigcirc}}
\renewcommand{\oplus}{\mathbin{\scriptstyle +\mkern-15.5mu\bigcirc}}
\renewcommand{\otimes}{\mathbin{\scriptstyle \times\mkern-15.5mu\bigcirc}}

\DeclareMathOperator{\GL}{GL}
\DeclareMathOperator{\SL}{SL}

\DeclareMathOperator{\GI}{GI}
\DeclareMathOperator{\SI}{SI}

\DeclareMathOperator{\Sp}{Sp}
\DeclareMathOperator{\GU}{GU}
\DeclareMathOperator{\GO}{GO}
\DeclareMathOperator{\U}{U}

\DeclareMathOperator{\SU}{SU}
\DeclareMathOperator{\SO}{SO}

\DeclareMathOperator{\PSL}{PSL}
\DeclareMathOperator{\PSp}{PSp}
\DeclareMathOperator{\PSU}{PSU}

\DeclareMathOperator{\diag}{diag}
\DeclareMathOperator{\rank}{rk}
\DeclareMathOperator{\id}{id}
\DeclareMathOperator{\rad}{rad}
\DeclareMathOperator{\sgn}{sgn}
\DeclareMathOperator{\chr}{char}

\begin{document}
\title{On the lattice of normal subgroups in ultraproducts of compact simple groups}
\author{Abel Stolz \and Andreas Thom}
\date{February 5, 2013}

\maketitle

\begin{abstract}We prove that the lattice of normal subgroups of ultraproducts of compact simple non-abelian groups is distributive. In the case of ultraproducts of finite simple groups or compact connected simple Lie groups of bounded rank the set of normal subgroups is shown to be linearly ordered by inclusion.

\end{abstract}

\tableofcontents

\section{Introduction}

This article is about the structure of the lattice of normal subgroups of an ultraproduct of compact simple groups. Note that the Peter-Weyl theorem implies that any compact simple group is either finite simple or a finite-dimensional connected compact simple Lie group; both cases admit a complete classification which is the basis of our considerations.

The motivation to study ultraproducts of groups is manifold. First of all, \emph{qualitative} properties of the ultraproduct reflect \emph{quantitative} properties of the groups involved. This becomes interesting since the manipulation of qualitative properties is sometimes easier and can be done with geometric or algebraic insight which might not be available in the quantitative computations. Let us give an example: We will later show that the set of normal subgroups of an ultraproduct of non-abelian finite simple groups is linearly ordered. Equivalently, there exists a natural number $k$, such that for each non-abelian finite simple group $G$ and each $g,h \in G$ either $g \in (C(h) \cup C(h^{-1}))^k$ or $h \in (C(g) \cup C(g^{-1}))^k$, where $C(s)$ denotes the conjugacy class of some element $s \in G$. While the qualitative statement sounds natural, the quantitative statement looks a bit more surprising at first. These statements can be proved only using the classification of finite simple groups.

Another source of of motivation for the study of normal subgroups of ultraproducts is the recent interest in sofic groups which has lead to the consideration of metric ultraproducts of symmetric groups. The connection between the two topics is provided by a theorem of Elek and Szab\'o \cite{elekszabo05}, asserting that a countable group is sofic if and only if embeds into a metric ultraproduct of symmetric groups. A \emph{metric} ultraproduct of groups is the quotient of an ultraproduct $\bs{G}$ by a normal subgroup $\bs{N}$, arising as the set of all elements infinitesimally close to the identity. In this context distance is measured in the Hamming distance on permutation groups. In constrast to the theory of sofic groups, where the above subgroup $\bs{N}$ is neglected, Ellis et~al.~\cite{ellisetal08} investigated this very normal subgroup, starting with the observation that it is maximal and hence $\bs{G}/\bs{N}$ simple. In fact they were able to show that the normal subgroups of $\bs{G}$ are linearly ordered by inclusion. Thus naturally the problem arose whether this theorem would generalize to ultraproducts of other (possibly all) non-abelian finite simple groups. We answer this question in the positive. The main source of knowledge used in the proof is Liebeck and Shalev's deep investigation of the size of conjugacy classes in finite simple groups \cite{liebeckshalev01}.
 Having thus dealt with all finite simple groups, where one can hope a priori for a positive answer, we take one more step to compact simple groups. In this setting an analogous theorem still holds true, under the somewhat restricting assumption of imposing a bound on the rank of groups contributing to the ultraproduct in question. When the bound on the rank is dropped the lattice of normal subgroups is no longer linearly ordered but we can still show that it is distributive, and in fact not very complicated. The method of proof takes its inspiration from seminal work of Nikolov and Segal \cite{nikolovsegal12}.

\paragraph*{}
The article is organized as follows. Section~\ref{sec:LengthFunctions} introduces notation and basic notions of metrics on permutation groups and matrix groups. Some effort is dedicated to the study of the connection of the Hamming distance to the size of conjugacy classes in symmetric groups. Although these results are not used in the sequel, they elucidate the above mentioned theorem of Ellis et~al. In fact one could easily reprove the theorem combining the results in Section~\ref{sec:LengthFunctions} and Section~\ref{sec:UltraproductsOfFiniteSimpleGroups}.

Section~\ref{sec:UltraproductsOfFiniteSimpleGroups} starts with some geometric prerequisites and culminates in Theorem~\ref{thm:NormalSubgroupsOfUltraproductsOfFiniteSimpleGroups}, validating the claim that the set of normal subgroups of ultraproducts of non-abelian finite simple groups is linearly ordered.

Section~\ref{sec:UltraproductsOfCompactConnectedSimpleLieGroups} consists mainly of an investigation of the inner structure of compact connected simple Lie groups. Then Theorem~\ref{thm:NormalSubgroupsOfUltraproductsOfLieGroups} for ultraproducts of compact connected simple Lie groups with bounded rank is deduced. Considering Lie groups of unbounded rank leads to Theorem~\ref{thm:DistributiveLatticeOfNormalSubgroupsOfUltraproducts}, asserting that the lattice of normal subgroups of ultraproducts of these is distributive.

In the concluding section the obtained results are bundled into Main Theorem~\ref{thm:MainTheorem}.

\paragraph*{}
The reader is assumed to be familiar with ultraproducts and ultralimits. Elementary properties of ultrafilters will be used without further notice. For a comprehensive introduction to (metric) ultraproducts and related notions confer \cite{benyaacovetal08}. We use more facts concerning finite simple groups of Lie type and Lie groups, respectively, than we are willing to introduce thoroughly. One may use the textbooks cited below or some standard reference of one's own choice, to verify missing links. Note that -- from now on -- when talking about finite simple groups we always mean finite simple \emph{non-abelian} groups.

\section{Length functions}\label{sec:LengthFunctions}
The study of groups is enriched when we introduce a compatible metric or topology. Metric or topological groups form a well understood subject of study. In this section we will introduce the notion of pseudo length functions, which in fact is just a reformulation of the notion of pseudometrics. We shall further give examples of pseudo length functions in some, mostly finite, groups and examine how different pseudo length functions can be compared in large groups.

We denote the set $\{1,\ldots,n\}$ of natural numbers by $[n]$.  In a group $G$ we write $g^G$ for the conjugacy class of an element $g$. The group generated by $g$ is $\left<g\right>$, the group generated by a subset $S\subset G$ is $\left<S\right>$, and consequently the normal subgroup generated by $g$ is $\left<g^G\right>$. When the group in which conjugation takes place is understood, we write $C(g)$ for the conjugacy class of $g$ and $N(g)$ for the normal closure $\left<g^G\right>$ of $g$. 
We call $S$ normal if it is the union of conjugacy classes and non-trivial if it contains a non-identity element.

Let $G$ be a group. A function $\ell:G\rightarrow [0,\infty[$ is called a \deph{length function} on $G$ if for all $g,h\in G$
\begin{itemize}
 \item[LF1] $\ell(g)=0$ if and only if $g=1$,
 \item[LF2] $\ell(g)=\ell(g^{-1})$,
 \item[LF3] $\ell(gh)\leq \ell(g)+\ell(h)$.
\end{itemize}
If moreover $\ell(hgh^{-1})=\ell(g)$ holds, then $\ell$ is \deph{invariant}. If the first axiom is weakened to $\ell(1)=0$, then $\ell$ is only a \deph{pseudo length function}.

It is an easy observation that every (pseudo) length function corresponds to a right-invariant (pseudo) metric on $G$ and vice versa by $d(g,h)=\ell(gh^{-1})$ and $\ell(g)=d(g,1)$. The notion of invariance for (pseudo) length functions translates into left-invariance of the corresponding (pseudo) metric. We say that a group $G$ has \deph{diameter} $D(G)$ with respect to $\ell$ if $\sup_{g\in G}\ell(g)=D(G)$. This notion coincides with the diameter of metric spaces.

It will turn out to be necessary to study the asymptotics of sequences of pseudo length functions on groups of growing size. Let $\mathcal{G}=\set{G_n}{n\in \mathbb{N}}$ be a countably infinite family of groups with generic length functions $\ell_1$ and $\ell_2$ defined for every $G\in \mathcal{G}$. We call $\ell_1$ \deph{asymptotically bounded} by $\ell_2$ if there are constants $c$ and $N$ such that for every $n\geq N$ and every choice of elements $g\in G_n$ we have $\ell_1(g)\leq c \ell_2(g).$
The constant $c$ is called a \deph{modulus} of asymptotic boundedness. The function $\ell_1$ is \deph{locally asymptotically bounded} by $\ell_2$ in \deph{radius} $\delta$, if the same holds for all $g \in G_n$ satisfying $\ell_1(g)<\delta$, for some $\delta>0$ not depending on $n$. We call $\ell_1$ and $\ell_2$ (\deph{locally}) \deph{asymptotically equivalent} if $\ell_1$ and $\ell_2$ are (locally) asymptotically bounded with respect to each other.

We are interested in the interaction of pseudo length functions and quotient groups. The following two lemmas introduce the natural definitions.

\begin{lemma}
Let $G$ be a finite group with a normal subgroup $H$ and an invariant (pseudo) length function $\ell$. Then 
\[\ell_{G/H}(gH):=\inf_{h\in H}\ell(gh)\]
defines an invariant (pseudo) length function on $G/H$.
\end{lemma}

\begin{proof}
We only show the triangle inequality. Let $g,h$ be in $G$ and $k,l$ in $H$ such that $\ell(gk)$ and $\ell(hl)$ are minimal. Then
\[\ell_{G/H}(ghH)\leq \ell(gkhl)\leq \ell(gk)+\ell(hl)=\ell_{G/H}(gH)+\ell_{G/H}(hH).\]
\end{proof}

The proof of the following statement is obvious.

\begin{lemma}
 Let $G$ be a group with normal subgroup $H\neq \{1\}$ and $\ell$ a (pseudo) length function on $G/H$. Then
\[\ell^G(g):=\ell(gH)\]
defines a pseudo length function on $G$. If $\ell$ is invariant, then $\ell^G$ is invariant, too.
\end{lemma}

\subsection{The conjugacy length}
An example of a pseudo length function that can be defined on any finite group $G$ is the \deph{conjugacy length}
\[\lc(g):=\frac{\log|C(g)|}{\log|G|}.\]

\begin{proposition}
Let $G$ be a finite group. Then the function $\lc$ is an invariant pseudo length function on $G$.
\end{proposition}

\begin{proof}
The claim follows from elementary properties of conjugacy classes. For example, $C(gh) \subset C(g)C(h)$ implies the triangle inequality.
\end{proof}

Note that $\lc$ is a length function if and only if $G$ has trivial center, in particular if $G$ is non-abelian and simple. More explicit is the following proposition.

\begin{proposition}\label{prop:ConjugacyLengthModuloCenter}
 Let $G$ be a finite group. Then
\[\lc(g)=\lc_{G/Z(G)}(gZ(G))\]
holds for all $g\in G$.
\end{proposition}

\begin{proof}
It is not hard to observe that $|C(gz)|=|C(g)|$ for any central element $z$, which proves
\[\lc_{G/Z(G)}(gZ(G))=\inf_{z\in Z(G)}\lc(gz)= \inf_{z\in Z(G)}\lc(g)= \lc(g).\]
\end{proof}

The following lemma is obvious.

\begin{lemma}
 Let $G$ be a finite group. Then for all $g\in G$ and $n\in\mathbb{N}$ the estimate
\[\lc(g^n)\leq \lc(g)\]
holds.
\end{lemma}

The conjugacy length is very useful, because it is directly related to algebraic properties of the group. We will make heavy use of results of Liebeck-Shalev. They showed in \cite{liebeckshalev01} that in any non-abelian simple group $G$, a conjugacy class of some element generates $G$ essentially as quickly as the conjugacy length permits. More precisely, Liebeck-Shalev showed that there is a constant $k$, such that $C(g)^{[k/ \lc(g)]} = G$ for all non-abelian finite simple groups $G$ and all $g \in G$. On the other side it is clear that at least $D(G)/\lc(g)$ products are necessary, and $D(G)$ is bounded below by a positive constant.
Hence, the result of Liebeck-Shalev is best possible. One drawback is that the conjugacy length is not directly related to geometry and sometimes hard to compute.
We will proceed by giving some examples of length functions on classes of groups from everyday life and show that for each finite simple group the conjugacy length can be replaced by more familiar invariant length functions related to geometry.

\subsection{Length functions on permutation groups}
We denote the class of all symmetric groups (i.e. full permutation groups of finite sets) by $\mathcal{S}$ and the class of alternating groups by $\mathcal{A}$.

\begin{proposition}\label{prop:RankLengthOnSn}
Let $\pi$ be a permutation in $S_n$ with $l$ cycles. Then
\[\lr(\pi):=1-\frac{l}{n}\]
defines an invariant length function on $S_n$.
\end{proposition}

We postpone the proof to Subsection~\ref{ssec:LengthFunctionsOnLinearGroups} and look at another example.

The \deph{Hamming length} of a permutation $\pi\in S_n$ is defined as 
\[\lh(\pi):=1-\frac{|\set{i\in [n]}{\pi(i) = i}|} {n} \]
It is well known that $\lh$ is an invariant length function on $S_n$.

The following proposition serves as an introductory example of asymptotic equivalence and will be useful later.

\begin{proposition}\label{prop:RankLengthAndHammingLengthOnSn}
The length functions $\lh$ and $\lr$ are asymptotically equivalent.
\end{proposition}

\begin{proof}
Let $\pi\in S_n$ with $l$ cycles, $m$ of which are trivial. Then immediately
\[\lr(\pi)=\frac{n-l}{n}\leq \frac{n-m}{n}=\lh(\pi)\] follows. Because the remaining $l-m$ non-trivial cycles have length at least $2$, $l-m\leq \frac{1}{2}(n-m)$. We conclude
\[\frac{n-m}{n}=\frac{n-l+l-m}{n}\leq \frac{n-l}{n}+\frac{n-m}{2n}\]
and finally $\lh(\pi)\leq 2 \lr(\pi)$.
\end{proof}

We shall use the remainder of this paragraph to exhibit the connection of the Hamming length to the generic conjugacy length introduced above.

\begin{lemma}\label{lem:SizeOfConjugacyClassInSn}
 Let $\pi$ be a permutation in $S_n$. If the number of cycles of length $i$ is denoted by $c_i$ and the longest cycle has length $k$, then the cardinality of the conjugacy class of $\pi$ in $S_n$ is given by
\[|C(\pi)|=n!\left(\prod_{i=1}^k i^{c_i} \prod_{i=1}^k c_i!\right)^{-1}.\]
\end{lemma}

\begin{proof}
The claim is elementary and follows by combinatorics as explained in \cite{wilson09}, Section~2.3.1.
\end{proof}

\begin{lemma}\label{lem:LengthFunctionComparisonNumberOneInSn}
The length function $\lc$ is asymptotically bounded by $\lh$ in $\mathcal{S}$.
\end{lemma}

\begin{proof}
We consider a non-trivial permutation $\pi\in S_n$ which has $m$ fixed points. Again we denote the number of cycles of length $i$ of $\pi$ by $c_i$. Then by assumption $c_1=m$. By Stirling's formula for large $n$ the estimate
\[\tfrac{1}{2}n\log n\leq\log n! \leq 2 n\log n\] holds.

Using Lemma~\ref{lem:SizeOfConjugacyClassInSn} we obtain the trivial inequality
\[|C(\pi)|\leq n!(m!)^{-1}.\]
Therefore
\[\lc(\pi) \leq \frac{\sum_{i=1}^n\log i-\sum_{i=1}^m \log i}{\frac{1}{2}n\log n}\leq 2 \frac{\sum_{i=m+1}^n\log n}{n\log n}=2\frac{n-m}{n}=2\lh(\pi).\]
\end{proof}

\begin{lemma}\label{lem:LengthFunctionComparisonNumberThreeInSn}
In $\mathcal{S}$ the length function $\lh$ is asymptotically bounded by $\lc$.
\end{lemma}

\begin{proof}
Let $\pi$ be a permutation in $S_n$. Assume that $\pi$ has $n-k$ fixed points, that is $\lh(\pi)=k/n$. For the sake of simplicity we only treat the case of even $k$ and note that the odd case is almost the same. We distinguish the cases $k>n/2$ and $k\leq n/2$.

If $k>n/2$ we can estimate the size of the centralizer of $\pi$ by
\[|C_{S_n}(\pi)|\leq n(n-1)\cdot \ldots \cdot (n-k/2),\]
since $\pi$ has at most $n-k/2$ cycles and every permutation commuting with $\pi$ is determined by its action on one point from each cycle of $\pi$. Therefore
\[|C(g)|\geq \frac{n!}{n(n-1)\cdot \ldots \cdot (n-k/2)}=(k/2)!.\]
 By loosely applying Stirling's approximation $(k/2)!\geq (k/2)^{k/4}$ follows. Since $\log(k/2)\geq \log n/2$ for $n\geq 17$
\[ \lc(\pi)=\frac{\log |C(\pi)|}{\log (n!)} \geq \frac{\log((k/2)!)}{\log(n!)}\geq \frac{\frac{1}{2}\log(k/2)\frac{k}{2}}{n\log n}\geq \frac{k}{8n}= \tfrac{1}{8}\lh(\pi).\]

If $k\leq n/2$, then $\pi$ has at most $k/2$ non-trivial cycles. Since the permutations commuting with $\pi$ are determined by the action of a single point from each cycle, we deduce the estimate
\[|C_{S_n}(\pi)|\leq (n-k)!k^{k/2}.\]
It is clear that $k^{k/2}\leq (n/2)^{k/2}$ and $n(n-1)\cdot \ldots \cdot (n-k+1)\geq (n/2)^k$, and therefore
\[|C(\pi)|\geq \frac{n!}{(n-k)!k^{k/2}} \geq \frac{n(n-1)\cdot \ldots \cdot (n-k+1)}{k^{k/2}}\geq \frac{(n/2)^k}{(n/2)^{k/2}}= (n/2)^{k/2}.\]
Because $2\log(n/2)\geq \log n$, we finally obtain
\[\lc(g)\geq \frac{\frac{k}{2}\log(n/2)}{\log(n!)} \geq \frac{k\log(n/2)}{2n\log n}\geq \tfrac{1}{4}\lh(\pi).\]
\end{proof}

\begin{theorem}\label{thm:EquivalentLengthFunctionsInSnAndAn}
In $\mathcal{S}$ and $\mathcal{A}$ the length functions $\lh$ and $\lc$ are asymptotically equivalent.
\end{theorem}

\begin{proof}
The conjugacy classes of $S_n$ behave in two different ways. Either they correspond to exactly one conjugacy class in $A_n$, or they split into two classes in $A_n$. In the first case the size of the conjugacy class stays the same, whereas in the second case it splits into two parts of equal size. (Confer \cite{wilson09}, Paragraph~2.3.2.) Now Lemma~\ref{lem:LengthFunctionComparisonNumberOneInSn} and Lemma~\ref{lem:LengthFunctionComparisonNumberThreeInSn} apply.
\end{proof}

\subsection{Length functions on linear groups}\label{ssec:LengthFunctionsOnLinearGroups}
Given a (finite dimensional) vector space $V$ we write $\GL(V)$ for all bijective linear transformations of $V$, $\SL(V)$ for all linear transformations of $V$ of determinant $1$. When $V=F^n$ for some field $F$ we use notation $\GL_n(F)$ and the like, which reduces further to $\GL_n(q)$ etc. when $F$ is the finite field $\mathbb{F}_q$ of order $q$.

We shall deal in particular with linear groups over finite fields and introduce the symbols $\mathcal{GL}(q)$ for the class of all general linear groups defined over the field $\mathbb{F}_q$ and $\mathcal{GL}$ for the union of these, where $q$ ranges over all prime powers. Exchanging general for special yields $\mathcal{SL}(q)$ and $\mathcal{SL}$. If $V$ is a vector space over a field $F$ we will write $1$ for the identity mapping $V\rightarrow V$ and write simply $\alpha$ for the mapping $\alpha \cdot 1$, where $\alpha\in F$.

\begin{proposition}
Let $V$ be a vector space of dimension $n$. Then
\[\lr(g):=n^{-1}\rank(1-g)\]
is an invariant length function on $\GL(V)$.
\end{proposition}

\begin{proof}
It is clear that $\lr$ takes its values in $[0,1]$ and that $\lr(g)=0$ if and only if $g=1$. Moreover, if $g,h$ are in $\GL(V)$, then
\[\rank(\id-g)=\rank(-g^{-1}(1-g))=\rank(1-g^{-1})\]
and
\begin{align*}
 \rank(1-gh)&=\rank((h^{-1}-g)h)\\
&=\rank((1-g)-(1-h^{-1}))\\
&\leq \rank(1-g)+\rank(1-h^{-1})\\
&=\rank(1-g)+\rank(1-h).
\end{align*}
The invariance of $\lr$ follows from
\[\rank(1-hgh^{-1})=\rank(h(1-g)h^{-1})=\rank(1-g).\]
\end{proof}

We call the function $\lr$ the \deph{rank length}.

Now the following conclusion is rather obvious.

\begin{proof}[Proposition~\ref{prop:RankLengthOnSn}]
The symmetric group embeds as the subgroup of permutation matrices into $\GL(V)$. If $\pi$ consists of the cycles $\pi_1,\ldots,\pi_l$ then the corresponding permutation matrix $P_\pi$ equals the direct sum $P_{\pi_1}\oplus \ldots \oplus P_{\pi_l}$, where $\rank(\id-P_{\pi_i})=k-1$ if $\pi_i$ has length $k$. Hence $\lr$ is the restriction of the rank length to permutations.
\end{proof}

We want to prove a similar result for general linear groups over finite fields as we obtained for permutation groups in the last subsection. As it turns out it is necessary to gain some independence of the base field. We therefore introduce the \deph{Jordan length} (the name of which is explained below.)
\[\lj:=(\lr_{\GL(V)/Z(\GL(V))})^{\GL(V)}.\]
A more explicit description of $\lj$ is
\[\lj(g)=n^{-1}\cdot \inf_{\alpha\in F^{\times}}\rank(\alpha - g),\]
as the center of $\GL(V)$ is isomorphic to $F^\times$.

From now on we shall write
\[m_g:=\sup_{\alpha\in F^{\times}}\dim(\ker(\alpha - g)),\]
whenever $g$ is an element in a linear group over a field $F$.
 With this definition yet another characterization of the Jordan length is
\[\lj(g)=\frac{n-m_g}{n}.\]

\begin{proposition}\label{prop:SmallRankLengthImpliesLargeJordanLength}
 Let $g$ be an element in $\GL(V)$. If $\lr(g)\leq \delta$, then $\lj(g)\geq \min\{(1-\delta),\delta\}$.
\end{proposition}

\begin{proof}
Let $m=\rank(1 -g)$. In the easiest case $\lj(g)=\lr(g)\geq \delta$. Hence we can assume $m\neq m_g$. Then of course $m+m_g\leq n$ and
\[\lj(g)=\frac{n-m_g}{n}\geq \frac{n-(n-m)}{n} = \frac{m}{n}=1-\lr(g) \geq 1- \delta\]
follows.
\end{proof}

\begin{corollary}
 Let $g$ be an element in $\GL(V)$. If $\lr(g)\leq 1/2$, then $\lr(g)=\lj(g)$.
\end{corollary}

\begin{proof}
 By definition $\lj(g)\leq \lr(g)$.
\end{proof}

We cite almost verbatim from the introduction of the Jordan decomposition on pp.~395, 396 in \cite{liebeckshalev01}. Each $g\in \SL_n(q)$ equals the commuting product $su$ of a semisimple element $s$ and a unipotent element $u$, this being called the Jordan decomposition. 

We denote by $J_k$ the unipotent $k\times k$ Jordan matrix
\[\left(\begin{matrix}
          1 & 1 &  & & \\
	  & 1 & 1 & &\\
	  &  & \ddots &\ddots & \\
	  && & 1 & 1\\
	  &&&& 1\\
        \end{matrix}\right)
\]
If $m_{ij}$ are non-negative integers for all $j=1\ldots r$, $i=1\ldots k_j$, we let $J^{m_j}:=m_{1j}J_1\oplus m_{2j}J_2\oplus \ldots \oplus m_{k_jj}J_{k_j}$, where $m J_i:=J_i\oplus \ldots \oplus J_i$ is the direct sum of $m$ Jordan blocks of size $i$. Then $g$ is conjugate to a matrix
\[\alpha_1J^{m_1}\oplus \ldots \oplus \alpha_rJ^{m_r}\oplus \lambda_1\otimes J^{n_1}\oplus \ldots \oplus \lambda_t \otimes J^{n_t},\]
where $m_j$ and $n_j$ are appropriate finite sequences of non-negative integers, $\alpha_j\in \mathbb{F}_q^\times$, and the $\lambda_j$ are irreducible matrices. In this representation we can assume that $m:=m_{11}\geq m_{12}\geq \ldots \geq m_{1r}$. That is $m$ counts the maximal number of Jordan blocks of size $1$ to an eigenvalue $\alpha_1\in\mathbb{F}_q$ of $g$.

Fortunately $m$ and $m_g$ compare well enough to relate the results in \cite{liebeckshalev01} to the Jordan length. A first application can be seen in the following theorem.

\begin{theorem}\label{thm:EquivalentLengthFunctionsInGeneralAndSpecialLinearGroups}
The pseudo length functions $\lc$ and $\lj$ are asymptotically equivalent in $\mathcal{GL}$ and $\mathcal{SL}$.
\end{theorem}

\begin{proof}
We first consider the case of special linear groups. Because $m$ counts the maximal number of Jordan blocks of size $1$ to a fixed eigenvalue of $g$, $m_g$ is maximal when we can find the maximal number of Jordan blocks of minimal size in the Jordan decomposition of $g$. The minimal size remaining for our choice is $2$ and hence
$m_g-m\leq \frac{n-m}{2}$. Now we can deduce as in the proof of Proposition~\ref{prop:RankLengthAndHammingLengthOnSn} that
\[\frac{n-m_g}{n}\geq \frac{n-m}{2n}\geq c'\frac{\log|C(g)|}{|SL_n(q)|}\]
for a constant $c'$, using also Lemma~5.3 in \cite{liebeckshalev01}. Thus $\lc$ is asymptotically bounded by $\lj$.

Lemma~5.4, ibid., states that there is a universal constant $c$ such that whenever $1\neq g\in\SL_n(q)$ and $k\geq \frac{cn}{n-m}$, then $C(g)^k=\SL_n(q)$. We can assume $c$ an integer and $k$ minimal such that $k=\epsilon c\frac{n}{n-m_g}\geq \epsilon c\frac{n}{n-m}$, where the error $\epsilon$ is definitely less than $2$. Now $|\SL_n(q)|\leq |C(g)|^k$ and 
\[k\lc(g)=\frac{\log(|C(g)|^k)}{\log|\SL_n(q)|}\geq \frac{\log|C(g)^k|}{\log|\SL_n(q)|}=1.\]
This implies
\[\lc(g)\geq k^{-1}=\epsilon^{-1}c^{-1}\frac{n-m_g}{n}\geq \tfrac{1}{2}c^{-1}\lj(g).\]

By comparison of the sizes of conjugacy classes in $\SL_n(q)$ and $\GL_n(q)$ the claim follows for $\mathcal{GL}(q)$. Because the above argument is independent of $q$, we have proved the theorem for $\mathcal{SL}$ and $\mathcal{GL}$.
\end{proof}

We remark that the proof of the preceding theorem, which is based on deep results of Liebeck-Shalev, can be done by elementary methods, estimating the size of centralizers of matrices with respect to the number and sizes of Jordan blocks involved. The elementary argument requires several pages, though, and therefore is omitted.

\section{Ultraproducts of finite simple groups}\label{sec:UltraproductsOfFiniteSimpleGroups}
We follow the notation in the first chapter of \cite{carter89} when adressing finite simple groups of Lie type. That is given a vector space $V$ with a bilinear or Hermitian form we write $\GI(V)$ for all isometries of $V$ and $\SI(V):=\GI(V)\cap \SL(V)$ (with the exception of orthogonal forms in characteristic $2$, which is explained below.) Thus for the trivial bilinear form $\GL(V)=\GI(V)$. We write $\Sp(V):=\SI(V)=\GI(V)$ for a symplectic bilinear form on $V$, $\GU(V):= \GI(V)$ and $\SU(V):=SI(V)$ for a Hermitian form on $V$ and $\GO(V):= \SI(V)$ and $\SO(V):= \SI(V)$ for a symmetric bilinear form on $V$ in odd characteristic. (Over characteristic $2$ the group $\SO(V)$ is defined as the kernel of the Dickson invariant.) Furthermore $\Omega(V):=\GO(V)'=\SO(V)'$, the commutator subgroup of $\SO(V)$. We prefix one more letter to denote the quotients of all these groups by their center, thus writing $\PSL(V)$, $\PSp(V)$ and so on.

 When dealing with ultrafilters we introduce the following abbreviating notation. Let $\mathfrak{u}$ be an ultrafilter on a set $I$. We say that a property $P$ holds $\mathfrak{u}$\deph{-almost everywhere} or for $\mathfrak{u}$\deph{-almost all} $i$ if the set $\set{i\in I}{P(i)}$ is in $\mathfrak{u}$. We also write $P(i)\alev{\mathfrak{u}}$ in this situation.

If $A_i$ is a family of algebraic structures indexed by $I$, we write $\prod_{i\rightarrow \mathfrak{u}}A_i$ for the ultraproduct or, when the right index is understood, only $\prod_{\mathfrak{u}}A_i$. A similar notation is used for limits along an ultrafilter, namely $\lim_{i\rightarrow\mathfrak{u}}a_i$, or $\lim_{\mathfrak{u}}a_i$ to save symbols.

In the following we fix a non-principal ultrafilter $\mathfrak{u}$ on the natural numbers.

 If $G_n$ are groups equipped with a generic (pseudo) length function $\ell$ we write $\bs{G}$ for the ultraproduct $\prod_{\mathfrak{u}}G_n$ and $\bs{g}$ for an element represented by a sequence $(g_n)_{n\in\mathbb{N}}$. Moreover we let  \[\ell(\bs{g}):=\lim_{\mathfrak{u}}\ell(g_n).\]
Then $\bs{N}:=\set{\bs{g}\in\bs{G}}{\ell(\bs{g})=0}$ is a normal subgroup, as can be deduced from the properties of pseudo length functions.

\begin{proposition}\label{prop:MetricUltraproductsOfFiniteSimpleGroups}
 Let $\mathcal{G}=\set{G_n}{n\in\mathbb{N}}$ be a collection of finite non-abelian simple groups. Then the group $\bs{G}/\bs{N}$ is simple.
\end{proposition}

\begin{proof}
We show that if $\lc(\bs{g})=\epsilon>0$ for $\bs{g}\in\bs{G}$, then already $N(\bs{g})=\bs{G}$. By Theorem~1.1 in \cite{liebeckshalev01} there is a universal constant $c$ such that whenever $G$ is a finite non-abelian simple group and $1\neq g\in G$, then $C(g)^m=G$ for any $m\geq c\frac{\log|G|}{\log|C(g)|}$. By our assumption $\frac{\log|G_{\omega(i)}|}{\log|C(g_i)|}\leq K\alev{\mathfrak{u}}$. Hence for $m\geq cK$, $C(g_i)^m=G_{\omega(i)}\alev{\mathfrak{u}}$ or equivalently $C(\bs{g})^m=\bs{G}$. We conclude that the set of all elements of zero length in $\bs{G}$ is a maximal normal subgroup and thus $\bs{G}$ divided by this subgroup is simple.
\end{proof}

In fact the converse is also true. If a quotient of a direct product of finite simple non-abelian groups is simple, then it is a quotient as in the preceding proposition for some choice of ultrafilter. Confer \cite{nikolov11}, proof of Proposition~3, for the argument.

\subsection{Some geometry}\label{ssc:GeometryOfFiniteQuasisimpleGroups}
We need some basic geometric lemmas to prepare what follows. We use the symbol $\operp$ to denote the orthogonal direct sum.  The next lemma is proved like Corollary~2.3 in \cite{grove02}.

\begin{lemma}\label{lem:LinearAndBilinearForms}
 Let $V$ be a finite dimensional vector space with non-degenerate bilinear or Hermitian form $(\cdot,\cdot)$ and $W$ some subspace. If $\phi\in W^*$, then there is $v\in V$ such that for all $w\in W$ the equation $(w,v)=\phi(w)$ holds.
\end{lemma}

\begin{lemma}\label{lem:ExtensionToNonDegenerateSubspace}
 Let $V$ be a finite dimensional vector space with non-degenerate bilinear or Hermitian form $(\cdot,\cdot)$ and $W$ a subspace. Let $R$ be the radical of $W$ and $W'$ a complement of $R$ in $W$. Then there is a subspace $W''$ of $V$, which satisfies $\dim(W'')=\dim(R)$ and $(W''\oplus W^\perp) \operp W'=V$. In particular $W'$ and $U:=W''\oplus W^\perp$ are non-degenerate.
\end{lemma}

\begin{proof}
We use Lemma~\ref{lem:LinearAndBilinearForms}. Let $w_1,\ldots,w_r$ be a basis of $R$ and $w_1,\ldots,w_k$ an extension to a basis of $W+W^\perp$. For $r=0$ there is nothing to show since $W\oplus W^\perp=V$. Now assume $r\geq 1$ and define $\phi_1\in (W+W^\perp)^*$ by $\phi_1(w_1)=1$ and $\phi_1(w_i)=0$ otherwise. Then there is $v_1\in V$ such that $(w_1,v_1)=1$ and $v\perp \left<w_2,\ldots,w_k\right>$. Now $\dim(\rad(W+W^\perp+\left<v_1\right>))=r-1$ and we can proceed inductively defining $\phi_l\in (W+W^\perp+\left<v_1,\ldots,v_{l-1}\right>)^*$ by $\phi_l(w_l)=1$ and $\phi_l(w_i)=0$, $\phi_l(v_i)=0$ for the remaining basis vectors. In the end this gives us $v_1,\ldots,v_r$ such that $W'':=\left<v_1,\ldots,v_r\right>$ meets our expectations.
\end{proof}

\begin{lemma}\label{lem:RestrictedGroupOfIsometries}
 Let $V$ be a finite dimensional vector space over a field $F$ with bilinear or Hermitian form $B=(\cdot,\cdot)$. We exclude the case that $\chr(F)=2$ and $B$ symmetric. Let $U$ be a non-degenerate subspace of $V$. Then the subgroup $H:=\set{g\in\SI(V)}{g|U^\perp=\id_{U^\perp}}$ of $\SI(V)$ is isomorphic to $\SI(U)$.
\end{lemma}

\begin{proof}
Let $g$ be in $\SI(U)$. We define $\phi(g):=g\oplus \id_{U^\perp}$. One verifies easily that $\phi$ maps $\SI(U)$ isomorphically onto $H$.
\end{proof}

\begin{lemma}\label{lem:RestrictedGroupOfOrthogonalIsometriesOne}
Let $V$ be a finite dimensional vector space over a field $F$ of odd characteristic with non-degenerate symmetric bilinear form $(\cdot,\cdot)$. Let $U$ be a non-degenerate subspace. Then the subgroup $H:=\set{g\in \Omega(V)}{g|U^\perp=\id_{U^\perp}}$ of $\Omega(V)$ is isomorphic to $\Omega(U)$.
\end{lemma}

\begin{proof}
Let $g\in H$. Because $g|U\in \GO(U)$, it can can be written as a product of reflections $s_{u_1},\ldots,s_{u_k}$ in $\GO(U)$, where the reflection is along the hyperplane $\left<u_i\right>^\perp$ (and $u_i$ non-degenerate). In particular $u_i\in U$ for all $i$.
Each reflection $s_{u_i}$ is given explicitly by the expression
\[s_{u_i}(w)=w-Q(u_i)^{-1}(w,u_i)u_i,\]
where $Q$ is the associated quadratic form. From orthogonality we deduce that $s^V_{u_i}:=s_{u_i}\oplus \id_{U^\perp}$ is a reflection in $\GO(V)$.

By Lemma~\ref{lem:RestrictedGroupOfIsometries} we know that $g|U\in \SO(V)$.
By \cite{grove02}, Theorem~9.7
\[\Omega(V)=\SO(V)\cap \ker(\theta),\]
where $\theta$ is the spinor norm. (Confer \cite{grove02}, Chapter~9, pp.~75, 76.) We see
\[1=\theta(g)=\theta(s^V_{u_1}\ldots s^V_{u_k})=Q(u_1)\cdot \ldots\cdot Q(u_k)F^{\times 2}=\theta(s_{u_1}\ldots s_{u_k})=\theta(g|U)\]
and conclude $g|U\in \Omega(U)$. The claim follows.
\end{proof}

\begin{lemma}\label{lem:RestrictedGroupOfOrthogonalIsometriesTwo}
 Let $V$ be a vector space of dimension $n$ over a perfect (e.g. finite) field of characteristic $2$. Let $Q$ be a regular (i.e. $\dim(\rad(V))\in\{0,1\}$) quadratic form on $V$ and $B=(\cdot,\cdot)$ the associated bilinear form. Let $U$ be a regular subspace of $V$ and $H:=\set{g\in \Omega(V)}{g|U^\perp=\id_{U^\perp}}$. Then $H$ is isomorphic to $\Omega(U)$.
\end{lemma}

\begin{proof}
We note that $H$ is isomorphic to a subgroup of $\GO(U)$ by $g\mapsto g|U$.  We divide the proof according to the dimension of the radical of $V$ and assume first that $V$ is defective. Then by Theorem~14.2 in \cite{grove02} $\GO(V)$ is isomorphic to $\Sp(V_1)$ for a complement $V_1$ of $\rad(V)$ and the action of $\GO(V)$ on $\rad(V)$ is trivial. We see that whether $U$ is defective or not, the proof of Lemma~\ref{lem:RestrictedGroupOfIsometries} applies.

 Now assume that $V$ is non-defective. By \cite{grove02}, Proposition~14.23 
\[\Omega(V)=\SO(V)\cap \ker(\theta),\]
where $\theta$ is the spinor norm. Note that $\SO(V)$ is the kernel of the Dickson invariant $\delta:\GO(V)\rightarrow \mathbb{F}_2$, or equivalently the subgroup of all products of an even number of orthogonal transvections. Now if $g\in H$, then $g|U\in \GO(U)$ and hence is a product of transvections $t_{u_1},\ldots, t_{u_k}$, see \cite{grove02}, Theorem~14.16. Each orthogonal transvection is described explicitly by
\[t_{u_i}(w)=w+Q(u_i)^{-1}(w,u_i)u_i.\]
We implicitly used that none of the vectors $u_i$ is singular. By extending $g|U$ to the whole of $V$ as in Lemma~\ref{lem:RestrictedGroupOfOrthogonalIsometriesOne}, $k$ is necessarily even. The proof now continues as in Lemma~\ref{lem:RestrictedGroupOfOrthogonalIsometriesOne}.
\end{proof}

\subsection{Normal subgroups of ultraproducts of finite simple groups}\label{ssec:NormalSubgroupsOfUltraproductsOfFiniteSimpleGroups}
In \cite{ellisetal08} the following result was proved.

\begin{theorem}[\cite{ellisetal08}, Theorem~1.1]\label{thm:TheoremOfEllisEtAl}
 Let $\mathfrak{u}$ be a non-principal ultrafilter on the natural numbers. Then the set of normal subgroups of $\prod_{\mathfrak{u}}A_n$ is linearly ordered.
\end{theorem}

Another formulation of this statement can be found in \cite{allsupkaye07} as Theorem~3.

As we start to develop the generalization of Theorem~\ref{thm:TheoremOfEllisEtAl} we recall the following standard principle, used when working with ultrafilters. Let $\bs{G}=\prod_\mathfrak{u}G_i$ be an ultraproduct of (not necessarily) groups, where $\mathfrak{u}$ is an ultrafilter on the index set $I$. If $U\in\mathfrak{u}$, then $\bs{G}$ is isomorphic to $\prod_{\mathfrak{u}|U}G_i$, where $\mathfrak{u}|U:=\set{J\cap U}{J\in\mathfrak{u}}$, and if $U_1\cup \ldots \cup U_k=I$, then there is $j$ such that $U_j\in \mathfrak{u}$. Hence combining both situations $\bs{G}$ is isomorphic to $\prod_{\mathfrak{u}|U_j}G_i$.

In our situation we can exploit this and treat different cases seperately. Given an ultraproduct $\bs{G}$ of arbitrary finite simple groups, the chosen ultrafilter 
``decides'' whether $\bs{G}$ is isomorphic to an ultraproduct of e.g. groups of bounded or unbounded rank, permutation groups or groups of Lie type, or in the case of groups of Lie type of large rank, which type of classical group they belong to. As a consequence there will be in particular no further treatment of sporadic groups and exceptional groups of Lie type apart from Proposition~\ref{prop:BoundedRankUltraproductIsSimple}, because they are finite or of bounded rank.

As a further example note also that we could replace the groups $A_n$ in Theorem~\ref{thm:TheoremOfEllisEtAl} by groups $A_{m_n}$. Then either $m_n\leq K$ for $\mathfrak{u}$-almost all $n$ and a constant $K$, and the ultraproduct itself is isomorphic to one $A_{m_n}$. Or $m_n\rightarrow_\mathfrak{u}\infty$ and in this case Theorem~\ref{thm:TheoremOfEllisEtAl} can be proved exactly as in \cite{ellisetal08}.

We obtain the following proposition for groups of Lie type almost instantly.

\begin{proposition}\label{prop:BoundedRankUltraproductIsSimple}
Let $G_n$ be finite simple groups of Lie type for all $n\in \mathbb{N}$. Let $\mathfrak{u}$ be a non-principal ultrafilter on the natural numbers. If $\bs{G}=\prod_{\mathfrak{u}}G_n$ and the rank of the groups $G_n$ is bounded, then $\bs{G}$ is simple.
\end{proposition}

 \begin{proof}
Suppose that the rank of the groups in question is bounded by $N$. Let $1\neq g\in G_n$. Using the constant $c$ of Theorem~1.1 in \cite{liebeckshalev01} we see that for
\[m\geq \frac{c\log|G_n|}{\log|C(g)|}\]
already $C(g)^m=G_n$. When $G_n$ is a group over the field $\mathbb{F}_q$, its order is at most $q^{c'N^2}$ for a universal constant $c'$. On the other hand a non-trivial conjugacy class in $G_n$ has at least $q$ elements. Hence it suffices for $m$ to be larger than $c'N^2$ to ensure $C(g)^m=G_n$ for any $n$.

If we choose $1\neq\bs{g}\in \bs{G}$ arbitrarily, then $C(g_n)^m=G_n$ for $\mathfrak{u}$-almost all $n$. Hence $C(\bs{g})^m=\bs{G}$ and consequently $N(\bs{g})=\bs{G}$. Therefore $\bs{G}$ contains no proper normal subgroups, whence it is simple.
\end{proof}

We take Theorems~\ref{thm:TheoremOfEllisEtAl} and Proposition \ref{prop:BoundedRankUltraproductIsSimple} as a motivation to prove the following  more general Theorem~\ref{thm:NormalSubgroupsOfUltraproductsOfFiniteSimpleGroups}. In the proof we follow a similar route as the authors of \cite{ellisetal08} in the proof of their Theorem~1.1.

\begin{theorem}\label{thm:NormalSubgroupsOfUltraproductsOfFiniteSimpleGroups}
Let $G_n$ be finite simple groups of Lie type for all $n\in \mathbb{N}$. If $\bs{G}:=\prod_{\mathfrak{u}}G_n$, then the set $\mathfrak{N}$ of normal subgroups of $\bs{G}$ is totally ordered.
\end{theorem}

In view of Proposition~\ref{prop:BoundedRankUltraproductIsSimple} we only need to take care of classical groups of unbounded rank.

First consider the general situation that we are given an ultraproduct $\bs{G}=\prod_{\mathfrak{u}}G_i$ of arbitrary groups $G_i$ with length function $\ell_i$. We define an ordering of the non-trivial elements of $\bs{G}$ by $\bs{g}\preceq \bs{h}$ if
\[\lim_{\mathfrak{u}} \frac{\ell_i(g_i)}{\ell_i(h_i)}<\infty.\]

\begin{lemma}\label{lem:OrderAndLengthConnection}
Let $\bs{g}$ and $\bs{h}$ be non-identity elements of the ultraproduct $\bs{G}$ of groups $G_i$. Then $\bs{g}\in N(\bs{h})$ implies $\bs{g}\preceq \bs{h}$. 
\end{lemma}

\begin{proof}
If $\bs{g}\in N(\bs{h})$ there is some integer $k$ such that $\bs{g}$ is a product of $k$ conjugates of $\bs{h}^{\pm 1}$. Therefore 
$g_i$ is a product of $k$ conjugates of $h_i^{\pm 1}$ for $\mathfrak{u}$-almost all $i$. By the properties of invariant length functions
\[\ell_i(g_i)\leq k\ell_i(h_i^{\pm 1}) =k\ell_i(h_i) \alev{\mathfrak{u}}.\]
Hence \[\lim_{\mathfrak{u}}\frac{\ell_i(g_i)}{\ell_i(h_i)}\leq k\] follows and we are done.
\end{proof}

We want to show that for finite simple groups of Lie type and the Jordan length the converse of the previous lemma is true.
The following statement is a summary of results from \cite{liebeckshalev01}.

\begin{lemma}\label{lem:BoundedGenerationInTermsOfJordanLength}
 Let $G$ be a quasisimple group of Lie type of rank $n$ and $g\in G\setminus Z(G)$. There is a constant $c$, independent of $G$ and $g$, such that $C(g)^m=G$ for all $m\geq \frac{cn}{n-m_g}$.
\end{lemma}

\begin{proof}
For special linear groups use Lemma~5.4 in \cite{liebeckshalev01}, for symplectic and orthogonal groups Lemma~6.4 and for unitary groups Section~7 ibid.
\end{proof}

\begin{lemma}\label{lem:OrderAndClassConnectionTwo}
Let $\bs{G}=\prod_{\mathfrak{u}}G_n$ be an ultraproduct of finite simple groups of Lie type equipped with the Jordan length. Then $\bs{g}\preceq \bs{h}$ implies $\bs{g}\in N(\bs{h})$ for all non-trivial elements $\bs{g},\bs{h}\in\bs{G}$.
\end{lemma}

\begin{proof}
Note that we can safely neglect exceptional groups, since those are of bounded rank and hence dealt with in Proposition~\ref{prop:BoundedRankUltraproductIsSimple}. More generally we assume that $m_n\rightarrow_{\mathfrak{u}}\infty$, where $m_n$ is the rank of the group $G_n$.
By the hypothesis there is a natural number $k$ such that $\frac{m_n-m_{g_n}}{m_n-m_{h_n}}\leq k$ for $\mathfrak{u}$-almost all $n$.

Let $G=\SI(V)$, where $V$ has dimension $n$. This group is only quasisimple, but working with the Jordan length will produce the same result in the ultra product. We can exclude the case when the characteristic of the field of definition is $2$ and $V$ is a defective quadratic space from the following considerations, since under that assumptions $G=\GO(V)$ is isomorphic to a symplectic group. Assume that $\frac{n-m_g}{n-m_h}\leq k$ for some non-trivial elements $g,h\in G$ such that $n-m_g=\rank(1-g)$ and $n-m_h=\rank(1-h)$, that is their rank length and Jordan length are the same. We define $W:=\ker(1-g)\cap \ker(1-h)$. If $W'$ is a complement of $\rad(W)$ in $W$, following Lemma~\ref{lem:ExtensionToNonDegenerateSubspace} there is subspace $W''$ such that $U:=W''\oplus W^\perp$ is non-degenerate and $W'=U^\perp$. Obviously $g$ and $h$ act as the identity on $U^\perp$. Then $g|U$ and $h|U$ are in $H:=\SI(U)$. We perform some calculations of dimensions to see
\[\dim(W^\perp) = n-\dim(W) \leq n-(m_g+m_h-n)=(n-m_g)+(n-m_h)\]
and
\[\dim (\rad(W))\leq n-\dim(W)\leq (n-m_g)+(n-m_h).\]
This together with the introductory remarks implies
\[\dim(U)=\dim(W^\perp)+\dim(\rad(W))\leq 2(n-m_g)+2(n-m_h)\leq (2k+2)(n-m_h). \]
Therefore the Jordan length of $h|U$ estimates as
\[\lj(h|U)=\frac{\dim(U)-(\dim(U)-(n-m_h))}{\dim(U)}=\frac{n-m_h}{\dim(U)}\geq \frac{1}{2k+2}.\]
By Lemma~\ref{lem:BoundedGenerationInTermsOfJordanLength} there is a constant $c$, independent of the hypotheses, such that $((h|U)^H)^m=H$ for $m\geq c(2k+2)$ and consequently $g|U$ is a product of $m$ conjugates of $h|U$ inside $H$. As in Lemma~\ref{lem:RestrictedGroupOfIsometries} we extend the elements occuring in this product to elements in $G$, thereby extending $g|U$ to $g$ and $h|U$ to $h$. Thus the conclusion remains true in $G$ and also when returning attention to the finite simple group $G/Z(G)$.

Because the prototype $G/Z(G)$ was independent of $n$ and the hypotheses did hold for almost all $n$, $g_n$ is a product of $m\geq c(2k+2)$ conjugates of $h_n$ in $G_n$ for almost all $n$.
Hence 
\[\bs{g}\in C(\bs{h})^m\subset N(\bs{h}),\]
which we had to prove.
\end{proof}

\begin{corollary}\label{cor:QuasiOrderOfGroupElements}
If $\bs{g}$ and $\bs{h}$ are non-identity elements in $\bs{G}$, the statements $\bs{g}\in N(\bs{h})$ and $\bs{g}\preceq \bs{h}$ are equivalent.
\end{corollary}

The last preparation we need is Lemma~2.2 in \cite{ellisetal08}, which for the sake of completeness we cite with proof.

\begin{lemma}\label{lem:OrderOfNormalSubgroupsAndOrderOfNormalClosures}
 Let $G$ be any group. Then the set of normal subgroups of $G$ is linearly ordered by inclusion if and only if the set of normal closures of non-identity elements in $G$ is.
\end{lemma}

\begin{proof}
The first implication is trivial. For the converse assume that $N$ and $M$ are normal subgroups of $G$ such that $N\not\subset M$. Let $g\in N\setminus M$ and observe that necessarily $N(g)\not\subset N(h)$ for all $h\in M$. Thus $N(h)\subset N(g)$ for all $h\in M$, and $M\subset N$ follows.
\end{proof}

\begin{proof}[Theorem~\ref{thm:NormalSubgroupsOfUltraproductsOfFiniteSimpleGroups}]
We define a quasiorder on the set $\bs{L}:=\prod_{\mathfrak{u}}[n]$ by $\bs{a}\preceq \bs{b}$ if $a_n\leq b_n$ for $\mathfrak{u}$-almost all $n$.
We let furthermore $\bs{a}\equiv \bs{b}$, whenever
\[0<\lim_{\mathfrak{u}}\frac{a_n}{b_n}<\infty.\]
Then $\equiv$ is a convex equivalence relation and the quotient space $\bs{L}/\equiv$ is totally ordered.

By the foregoing considerations, culminating in Corollary~\ref{cor:QuasiOrderOfGroupElements}, the set of normal closures of elements in $\bs{G}$ is order isomorphic to $\bs{L}/\equiv$. Lemma~\ref{lem:OrderOfNormalSubgroupsAndOrderOfNormalClosures} shows that the set of normal subgroups of $\bs{G}$ is linearly ordered by inclusion if and only if the set of normal closures of elements of $\bs{G}$ is. Now Theorem~\ref{thm:NormalSubgroupsOfUltraproductsOfFiniteSimpleGroups} follows.
\end{proof}

After the main theorem explaining the ordering of normal subgroups in ultraproducts of finite simple groups is established, we take a closer look at those.

\begin{lemma}
 Let $\bs{G}$ be an ultraproduct of finite simple groups. A normal subgroup $N\subset \bs{G}$ is of the form $N(\bs{g})$ for some $\bs{g}\in \bs{G}\setminus\{1\}$ if and only if $N$ has a predecessor with respect to the ordering of normal subgroups. 
\end{lemma}

\begin{proof}
 Using Theorem~\ref{thm:NormalSubgroupsOfUltraproductsOfFiniteSimpleGroups}, it is easy to see that the set $\set{\bs{h}\in\bs{G}}{\bs{h}\prec \bs{g}}$ is a maximal normal subgroup of $N(\bs{g})$. Conversely, let $N$ be a normal subgroup of $\bs{G}$ with a predecessor $N_0$. Then there exists $\bs{g}\in N\setminus N_0$. Since $N_0\subsetneq N(\bs{g})\subset N$, we conclude $N=N(\bs{g})$, since $N_0$ is the predecessor of $N$.
\end{proof}

If $\bs{g}\in\bs{G}$ we denote the predecessor of $N(\bs{g})$ by $N_0(\bs{g})$.

\begin{proposition} \label{prop:comm}
 Let $\bs{G}$ be an ultraproduct of finite simple groups. Then every normal subgroup $N$ in $\bs{G}$ is perfect. Indeed every element in $N$ is itself a commutator of elements in $N$.
\end{proposition}

\begin{proof}
 For a start assume that $\bs{G}$ is an ultraproduct of alternating groups. Given $\bs{g}$ we assume $g_n$ is an element in $A_{m_n}$. We can consider $g_n$ as an element in the alternating group, $A_k$ say, of the support of $g_n$. The famous paper \cite{ore51} of Ore (which led to the Ore Conjecture) implies that $g_n$ is a commutator of elements $x_n$ and $y_n$ in $A_k$, as long as $k\geq 5$, which we can assume without worry. Interpreting $x_n$ and $y_n$ as elements in $A_{m_n}$ we automatically have $\lh(x_n),\lh(y_n)\leq \lh(g_n)$. Therefore $\bs{x},\bs{y}\preceq \bs{g}$, which entails $\bs{x},\bs{y}\in N(\bs{g})\subset N$.

 In the case of groups of Lie type we go the same way.
We have to use the Ore conjecture, solved by Liebeck et al. in \cite{liebecketal10}. If $g_n\in \SL_n(q)$ is in the preimage of $\overline{g_n}\in\PSL_n(q)$ such that $\lr(g_n)=\lj(g_n)$, then $\ker(1-g_n)$ has a higher dimension than $\ker(\alpha-g_n)$ for any other $\alpha\in\mathbb{F}_q\setminus\{1\}$. We fix a complement of $\ker(1-g_n)$ and call this subspace $U$. It is clear that $U$ is invariant under the action of $g_n$ and $g_n|U\in \SL(U)$. Hence we find $\overline{x_n}$ and $\overline{y_n}$ in $\PSL(U)$ such that $\pi(g_n|U)=[\overline{x_n},\overline{y_n}]$, where $\pi$ projects onto $\PSL(U)$. Let $x_n$ and $y_n$ be in the preimage of $\overline{x_n}$ and $\overline{y_n}$, respectively. We can assume $\lj(x_n)=\lr(x_n)$ and $\lj(y_n)=\lr(y_n)$. Then $\lj(x_n\oplus 1),\lj(y_n\oplus 1)\leq \lj(g_n)$ and $\lj([x_n\oplus 1,y_n\oplus 1])=\lj(g_n)$. Hence if we pass to the ultraproduct $\bs{x},\bs{y},[\bs{x},\bs{y}]\in N$.

Now assume $g_n$ belongs to a symplectic, orthogonal or unitary group $\SI_n(q)$. As made clear above we are free to assume $\lj(g_n)=\lr(g_n)$. We use the geometric considerations in Subsection~\ref{ssc:GeometryOfFiniteQuasisimpleGroups}, especially Lemma~\ref{lem:ExtensionToNonDegenerateSubspace}. Let $W:=\ker(1-g_n)$ and $W'$ a complement of $\rad(W)$ in $W$. We obtain a non-degenerate subspace $U$ in $\mathbb{F}_q^n$ such that $g_n$ acts as the identity on $U$, $U^\perp=W'$ and $\dim(U^\perp)\leq 2(n-\dim \ker(1-g_n))$. Thus we can restrict $g_n$ to $U^\perp$ and proceed as above.
\end{proof}

\begin{corollary}
 Let $\bs{G}$ be an ultraproduct of finite simple groups, $\bs{g}\in\bs{G}\setminus\{1\}$ and $N$ a proper normal subgroup in $N(\bs{g})$. Then every element in the group $N(\bs{g})/N$ is a commutator.
\end{corollary}
By the maximality of predecessors we deduce one more corollary.

\begin{corollary} \label{cor:simple}
 If $\bs{G}$ is an ultraproduct of finite simple groups, then the group $N({\bs{g}})/N_0(\bs{g})$ is perfect and simple for all $\bs{g}\in\bs{G}\setminus\{1\}$.
\end{corollary}

In order to prove the previous corollary, it is enough to assume in the proof of Proposition \ref{prop:comm} that there exists a universal constant $c$, such that every element in the commutator subgroup of a non-abelian quasisimple group is a product of at most $c$ commutators. This was established by Wilson \cite{wilson} long before the Ore Conjecture was solved.

\section{Ultraproducts of compact connected simple Lie groups}\label{sec:UltraproductsOfCompactConnectedSimpleLieGroups}
We want to show that an analogue of Theorem~\ref{thm:NormalSubgroupsOfUltraproductsOfFiniteSimpleGroups} holds for quasisimple Lie groups. 

\subsection{Bounded generation in compact connected simple Lie groups}
The motivation for this paragraph is taken from \cite{nikolovsegal12}, Paragraph~5.5.4, wherefrom we freely cite all we need. The goal is to refine the methods from ibid. to obtain the result that in compact connected simple Lie groups an element which is not much longer, in a certain sense, than some other element, can be written as a bounded product of conjugates of the latter.

 Let $G$ be a compact connected quasisimple Lie group. (That is a compact connected perfect Lie group, which is simple modulo its centre.) Then $G$ contains a maximal compact connected abelian subgroup $T$, called a maximal torus of $G$. The dimension as a manifold of $T$ is called the rank of $G$. The choice of a torus $T$ is unique up to conjugation and every element in $G$ is conjugate to one in $T$. Assume that the rank of $G$ is $r$. Then there is a set $\Phi=\{\beta_1,\ldots,\beta_r\}$ of fundamental roots, determining $T$. Each root $\alpha$ corresponds to a character $\alpha:T\rightarrow S^1$. Then
\[\bigcap_{i=1}^r\ker \beta_i=Z(G).\]
A complex number $\mu$ in $S^1$ can be written uniquely as $e^{\imag\theta}$, where $\theta\in ]-\pi,\pi]$. We call $l(\mu):=|\theta|$ the \deph{angle} of $\mu$. Now we define
\[\lt(g):=\frac{1}{\pi r}\sum_{i=1}^rl(\beta_i(g))\]
for all $g\in T$. 
 Proposition~5.11 in \cite{nikolovsegal12}, the proof of which is spread over Subsection~5.5 ibid., includes the following result.

\begin{proposition}\label{prop:TorusLengthFunction}
 The function $\lt:T\rightarrow \mathbb{R}$ is an invariant pseudo length function on $T$ and $\lt(g)=0$ if and only if $g\in Z(G)$.
\end{proposition}

If not explicitly stated otherwise we will safely assume that $G$ is a simply connected Lie group, since $\lt$ is zero on the center of $G$ and thus well defined on the quotient $G/Z(G)$. We continue with further features of the internal structure of a (simply connected) compact connected quasisimple Lie group, as outlined in \cite{nikolovsegal12}. We make adjustments to the text and notation of this reference when needed.

For every character $\alpha$ we have a cocharacter $\eta_\alpha:S^1\rightarrow T$ such that $\alpha(\eta_\alpha(\mu))=\mu^2$ for all $\mu\in S^1$. For every pair of opposite roots $\pm \alpha$ of $\Phi$ there is a homomorphism $\phi_\alpha:\SU(2)\rightarrow G$ such that $\eta_\alpha$ is the restriction of $\phi_\alpha$ to the subgroup of diagonal matrices of $\SU(2)$.
For every root $\alpha$ we define subgroups associated with it. There is $T_\alpha:=\set{g\in T}{\alpha(g)=1}\subset T$. Then we have $S_\alpha=S_{-\alpha}$, the image of $\SU(2)\subset G$ under $\phi_\alpha$. $S_\alpha$ commutes elementwise with $T_\alpha$ and $T$ is contained in the central product $S_\alpha T_\alpha$. At last we define the one parameter torus $H_\alpha$ as the image of the cocharacter $\eta_\alpha$. For fundamental roots $\beta_i$ we use the self-explanatory shorthand notation $T_i$, $S_i$ and $H_i$. Then $T$ equals the direct product $H_1H_2 \ldots H_r$. 

Every element $g\in T$ can be decomposed into the product of commuting factors $g=g_1\cdot \ldots \cdot g_r$, where $g_i\in H_i$. We define
\[g'_i:=g_1\cdot \ldots\cdot g_{i-1}g_{i+1}\cdot \ldots \cdot g_r.\]
Then it is clear that $g=g_ig'_i=g'_ig_i$ for any $i$. Moreover $l(\beta_i(g_i))=l(\beta_i(g))$ and $l(\beta_i(g'_i))=0$, since $g'_i\in T_i$. 

The following result can be found in the proof of Lemma~5.20 in \cite{nikolovsegal12}.

\begin{lemma}\label{lem:BoundedGenerationInSUTwo}
 Let $G=\SU(2)$ and $g,h$ be non-trivial elements in $G$ such that $\lt(g)\leq m\lt(h)$, $m\geq 2$ an integer. Then $g$ is a product of at most $m$ conjugates of $h$.
\end{lemma}

We give two lemmas concerning linear combinations of roots, and the resulting impact on the structure of $G$.

\begin{lemma}\label{lem:LinearCombinationsOfRoots}
 In any simple root system $\Phi$ every long root $\beta$ can be written as $\beta=\alpha_1+\alpha_2$ for short roots $\alpha_1$ and $\alpha_2$. Every short root $\alpha$ can be written as $\alpha=\mu\beta_1+\mu\beta_2$, where $\mu\in \pm\{1/3, 1/2, 1\}$. These are the only coefficients that can appear in a linear combination of two roots to a third.
\end{lemma}

\begin{proof}
 The lemma follows from inspection of the standard representations of root systems.

If $\Phi$ is of type $A_n$, $D_n$, $E_6$, $E_7$ or $E_8$ all roots have the same length and there is nothing to prove. In case of $\Phi$ being of type $B_n$, the roots are exactly the integer vectors $v$ in $\mathbb{R}^n$ with Euclidean norm $|v|=1$ or $|v|=\sqrt{2}$. For type $C_n$ we have $\Phi=\set{v\in\mathbb{Z}^n}{|v|=\sqrt{2}}\cup \set{v\in(2\mathbb{Z})^n}{|v|=2}$. We see that in these cases $\mu=\pm 1/2$.
If $\Phi$ is of type $F_4$, it is the union of the set of all vectors in $\mathbb{R}^4$ with two or one components equal to $\pm 1$ and the others equal to $0$ and the set of vectors with all components being $\pm 1/2$. Here $\mu$ is either $\pm 1/2$ or $\pm 1$, depending on the short root. 
In the remaining case of type $G_2$ we represent $\Phi$ by vectors in $\mathbb{R}^3$, the short roots being
\[(1,-1,0),\, (-1,1,0),\, (1,0,-1),\, (-1,0,1),\, (0,1,-1),\, (0,-1,1)\]
and the long roots
\[(2,-1,-1),\, (-2,1,1),\, (1,-2,1),\, (-1,2,-1),\, (1,1,-2),\, (-1,-1,2).\]
Again a close look implies the claim with $\mu=\pm 1/3$.
\end{proof}

\begin{lemma}\label{lem:ProductsOfOneDimensionalTori}
Let $\alpha$, $\beta$ be fundamental roots of different lengths and $g$ an element in $H_\alpha$ such that $l(\alpha(g))=\epsilon$. Then there are elements $w_1$ and $w_2$ in the Weyl group such that $H_\alpha\subset H_\beta^{w_1}H_\beta^{w_2}$ and in particular $g$ equals the product $g_1g_2$, where $g_i\in H_\beta^{w_i}$ are elements such that $l(\beta^{w_i}(g_i))\leq \epsilon$.
\end{lemma}

\begin{proof}
 The inclusion $H_\alpha\subset H_\beta^{w_1}H_\beta^{w_2}$ can be found in \cite{nikolovsegal12} in the proof of Lemma~5.19 and the argument goes as follows. The Weyl group $W$ acts on the roots. There are elements $w_1$ and $w_2$ in $W$ such that $\alpha$ equals the linear combination $\mu_1\beta^{w_1}+\mu_2\beta^{w_2}$. The claim follows.

We have to go into detail and take care of lengths of elements in the product. To each root $\delta$ corresponds a coroot $h_\delta$ in the Lie algebra of $G$. Chapter~23 in \cite{bump04} shows that there is a normalization of coroots such that we can assume $h_{\delta+\zeta}=h_\delta+h_\zeta$. The homomorphism $\eta_\delta$ is induced by $e_\delta:\theta\mapsto \exp(\theta \imag h_\delta)$. (Confer \cite{sepanski07}, Theorems~6.20, 4.8, 4.16.) Here the angle $l(\delta(e_\delta(\theta)))$ equals $2|\theta|$ if $|\theta|\in[0,\frac{1}{2}\pi]$ and $2\pi-2|\theta|$ if $|\theta|\in [\frac{1}{2}\pi,\pi]$.  From Lemma~\ref{lem:LinearCombinationsOfRoots} we know that the coefficients $\mu:=\mu_1=\mu_2$ are $\pm\frac{1}{3}$, $\pm\frac{1}{2}$ or $\pm1$. Hence $\mu^{-1}$ is an integer and if we write $\gamma_i:= \beta^{w_i}$, $i=1,2$, 
\[h_{\gamma_1}+h_{\gamma_2}=h_{\gamma_1+\gamma_2}=h_{\mu^{-1}\alpha}=\mu^{-1}h_{\alpha}\]
and so the coroots obey the same linear relation as the roots.

Assume without loss of generality $g=e_\alpha(\frac{1}{2}\epsilon)$, where $\epsilon<\pi$. Then $l(\alpha(g))=\epsilon$ and $\alpha=\mu\gamma_1+\mu\gamma_2$ implies
\[e_\alpha(\tfrac{1}{2}\epsilon)=\exp(\tfrac{1}{2}\epsilon\imag h_\alpha)=\exp(\tfrac{1}{2}\mu\epsilon \imag h_{\gamma_1})\exp(\tfrac{1}{2}\mu\epsilon \imag h_{\gamma_2})=e_{\gamma_1}(\tfrac{1}{2}\mu\epsilon)e_{\gamma_2}(\tfrac{1}{2}\mu\epsilon).\]
Hence $g=e_\alpha(\frac{1}{2}\epsilon)\in H_\alpha$ is the product of elements $g_i=e_{\gamma_i}(\frac{1}{2}\mu\epsilon)$  in the subgroups $H_\beta^{w_1}$ and $H_\beta^{w_2}$, respectively, with angle $l(\gamma_i(g_i))=\mu\epsilon\leq \epsilon$.
\end{proof}

We are now ready to generalize Lemma~\ref{lem:BoundedGenerationInSUTwo} to an arbitrary compact connected quasisimple Lie group $G$. We use the interplay of the groups $S_i$ and the Weyl group $W$, and the decomposition of the maximal torus $T$ into subgroups $H_i$. 

\begin{lemma}\label{lem:BoundedGenerationOfTorusParts}
 Let $g_i\in H_i$ and $h_j\in H_j$, corresponding to $g$ and $h$ in $T$, such that 
$l(\beta_i(g))\leq ml(\beta_j(h))$, where $m$ is an even integer. Then $g_i\in (h^{G}\cup h^{-G})^{4m}$.
\end{lemma}

\begin{proof}
The proof splits in two cases whether $\beta_i$ and $\beta_j$ have the same length or not.

If $\beta_i$ and $\beta_j$ are roots of the same length, then there is an element $v$ in the Weyl group $W$ such that $H_i^v=H_j$. We see that this entails $g_i^v\in H_j\subset S_j$ and, because the action of the Weyl group on the roots is by conjugation of the argument, $l(\beta_j(g_i^v))=l(\beta_i(g_i))$. Now $g_i^v\in (h_j^{S_j})^m$ by Lemma~\ref{lem:BoundedGenerationInSUTwo}.

We compute \[(h^{S_j})^m=((h_jh'_j)^{S_j})^m=(h_j^{S_j})^m\cdot {h'_j}^m\]
to deduce \[g_i\in ((h^{S_j})^m{h'_j}^{-m})^{v^{-1}}.\]

Now $l(\beta_j(1))=0$ and by Lemma~\ref{lem:BoundedGenerationInSUTwo}, $1\in (h_j^{S_j})^2$. Therefore, and because $h'_j$ commutes with every element in $S_j$,
\[{h'_j}^2\in (h_j^{S_j})^2 \cdot (h_j')^2 = (h^{S_j})^2.\]
Note that this works equally well for $h^{-1}$ instead of $h$. Because we assumed $m$ even, we arrive at
\[g_i\in ((h^{S_j})^m(h^{-S_j})^m)^{v^{-1}}\subset (h^{G}\cup h^{-G})^{2m}.\]

If $\beta_i$ and $\beta_j$ are roots of different lengths, Lemma~\ref{lem:ProductsOfOneDimensionalTori} gives the existence of elements $w_1$ and $w_2$ such that $g_i=f_1^{w_1}f_2^{w_2}$, where $f_k$ is in $H_j$ and $l(\beta_j(f_k))\leq l(\beta_i(g_i))$, for $k=1,2$. Then, again by Lemma~\ref{lem:BoundedGenerationInSUTwo}, $f_k\in (h_j^{S_j})^m$, and we obtain
\[g_i\in ((h_j^{S_j})^m)^{w_1}\cdot ((h_j^{S_j})^m)^{w_2}.\]
We now proceed as above to deduce
\[g_i\in ((h^{S_j})^m(h^{-S_j})^m)^{w_1} \cdot ((h^{S_j})^m(h^{-S_j})^m)^{w_2}\subset (h^{G}\cup h^{-G})^{4m}.\]
\end{proof}

The next theorem is modelled after Case~1 in Lemma~5.19 in \cite{nikolovsegal12}.

\begin{theorem}\label{thm:BoundedGenerationOfSmallElementsInLieGroupsOfSmallRank}
 Let $\epsilon>0$ and $G$ be a compact connected simple Lie group of rank $r$. Assume $g$ and $h$ are non-trivial elements in $T$ satisfying $\lt(h)=\epsilon$ and $\lt(g)\leq m\lt(h)$ for an even integer number $m$. Then \[g\in (h^G\cup h^{-G})^{4mr^2}.\]
\end{theorem}

\begin{proof}
 Write $g=g_1\cdot \ldots\cdot g_r$, $h=h_1\cdot \ldots \cdot h_r$, where $g_i,h_i\in H_i$. For reasons of averaging there is one fundamental root $\beta_j$ such that $l(\beta_j(h))\geq \epsilon\pi$. Let $m_i\geq 2$ be the smallest even integer such that $\lt(g_i)\leq m_i\lt(h_j)$. Then $m_i$ cannot be larger than $mr$ for any $i$. Therefore we get $\sum_{i=1}^rm_i\leq mr^2$. We now use Lemma~\ref{lem:BoundedGenerationOfTorusParts} to obtain for all $i$
\[g_i\in (h^G\cup h^{-G})^{4m_i},\]
independently of the length of roots involved. Because $g$ is the product of the $g_i$, and summing the $m_i$ gives at most $mr^2$, $g$ is a product of $4mr^2$ or less conjugates of $h$ and $h^{-1}$.
\end{proof}

\subsection{Normal subgroups of ultraproducts of compact connected simple Lie groups}
\begin{proposition}\label{prop:LOneLengthOnUnitaryGroups}
 Let $\|u\|$ denote the $l^1$-norm on the matrix ring $M_n(\mathbb{C})$. Then $\ell_1(u):=\frac{1}{2n}\|1-u\|$ defines an invariant length function on the group of unitary complex matrices $\U(n)$.
\end{proposition}

\begin{proof}
This follows from well known properties of unitary groups and matrix norms.
\end{proof}

In the following we are going to use fixed unitary representations of different types of Lie groups, which we will refer to as standard representations. (Confer \cite{bump04}, Chapter~20.) We use the obvious embedding of $\SU(n)$ in $\U(n)$. We have $\Sp(2n)$ realized inside $\U(2n)$ as matrices of the form $\left(\begin{smallmatrix} a &-\overline{b}\\ b &\overline{a}\end{smallmatrix}\right)$ with complex $n\times n$-matrices $a,b$. The orthogonal matrices $\SO(2n+1)$ embed into $\U(2n+1)$ such that their maximal torus consists of elements $\diag(t_1,\ldots, t_n,1,t_n^{-1},\ldots,t_1^{-1})$. The situation is similar for $\SO(2n)$ except the $1$ in the middle is missing. For the exceptional groups $E_7$, $E_8$, $F_4$ and $G_2$ let $\delta$ be the smallest fundamental representation and for $E_6$ the second smallest fundamental representation, of maximal dimension $351$ in the case of $E_6$. As the standard representation we use $\delta':g\mapsto \delta(g)\oplus \bar\delta(g)$.

Note that we have to consider exceptional Lie groups since now in the bounded rank case instead of Proposition~\ref{prop:BoundedRankUltraproductIsSimple} we are facing the more complicated Theorem~\ref{thm:NormalSubgroupsOfUltraproductsOfLieGroups} below. 

We define
\[\tilde{\lt}(g):=\sup_{t\in C(g)\cap T}\lambda(t).\]
 This new function has the advantage that, in contrast to $\lambda$, it is invariant under conjugation. Moreover, as can be expected it usually takes considerably larger values than $\lambda$, a fact we shall exploit in the proof of the next lemma.

\begin{lemma}\label{lem:LOneLengthOnCompactLieGroups}
  Let $G$ be a compact connected quasisimple Lie group of rank $r$ contained in a unitary group $\U(n)$ by the respective standard embedding. We write 
 \[\ell_1'(g):=\inf_{z\in Z(U(n))}\frac{n}{r}\ell_1(zg)\]
 for elements $g\in G$. Then there is a constant $L$ such that the following holds: If $G$ is classical, then for any $g$ in $G$
 \[\frac{1}{L}\ell_1'(g)\leq \tilde{\lt}(g)\leq L\ell_1'(g).\]
 If $G$ is exceptional, for every $g\in G$ with $\tilde{\lt}(g)\leq \frac{1}{2r}$
 \[\frac{1}{L}\ell_1'(g)\leq \lt(g)\leq L\ell_1'(g).\]
\end{lemma}

\begin{proof}
 In $\U(n)$ we can write
\[\ell_1(g)=\frac{1}{2n}\sum_{i=1}^n|1-\mu_i|,\]
 where $\mu_i$ are the eigenvalues of $g$. For $\theta\in [0,\pi]$ we have $|1-e^{\imag \theta}|=\sqrt{2}\sqrt{1-\cos \theta}$. By some analysis it can be seen that $1-\cos\theta\leq 8\frac{\theta^2}{\pi^2}$ and hence $\frac{1}{4}\sqrt{2}\sqrt{1-\cos\theta}\leq \theta/\pi$. We also have $\theta/\pi\leq \frac{1}{2}\sqrt{2}\sqrt{1-\cos\theta}$. By symmetry we obtain the necessary estimates for $\theta\in]-\pi,0]$. Thus
\[l(e^{\imag\theta})/\pi\leq \tfrac{1}{2}|1-e^{\imag\theta}|\leq 2l(e^{\imag\theta})/\pi.\]

Let first $G$ be equal to $\SU(n)$. For diagonal elements $t=\diag(t_1,\ldots,t_n)$ in the torus $T_n$ of diagonal matrices of determinant $1$ we have $\beta_i(t)=t_it_{i+1}^{-1}$. Therefore (abusing notation to apply $\beta_i$ to elements not in $\SU(n)$) $\lt(t)=\lt(zt)$ for any central element $z=\diag(z,\ldots ,z)\in Z(U(n))$. 
Because $\ell_1$ on $S^{1}$ is a length function $|1-t_it_{i+1}^{-1}|\leq |1-t_i|+|1-t_{i+1}|$, and hence
\[|1-zt_i(zt_{i+1})^{-1}|\leq |1-zt_i|+|1-zt_{i+1}|.\]
 Therefore the estimate
\[\lt(t)=\inf_{z\in Z(\U(n))}\lt(zt)\leq \inf_{z\in Z(\U(n))}\frac{1}{n-1}\sum |1-zt_i|=2\ell_1'(t)\]
follows. Then also $\tilde{\lt}(g)\leq 2\ell_1'(g)$ holds for any $g\in \SU(n)$.
Given $t$ we can reorder its diagonal entries by conjugation with a generalized permutation matrices (permutation matrices with entries in $\pm 1$ and determinant $1$), such that without loss of generality $l(t_1t_2^{-1})$ is maximal among all possible values $l(t_it_j^{-1})$. Proceeding from this point we can achieve inductively that $l(t_it_{i+1}^{-1})\geq l(t_it_j^{-1})$ for all $j>i$.
This yields $l(t_it_n^{-1})\leq l(t_it_{i+1}^{-1})$ for all $i=1\ldots n-1$.
Now
\[\inf_{z\in Z(\U(n))}\sum_{i=1}^n l(t_iz)\leq \sum_{i=1}^n l(t_it_n^{-1})\leq \sum_{i=1}^{n-1} l(t_it_{i+1}^{-1})\] follows. After normalizing the sums with the right factor and with the introductory estimates, $\ell_1'(t)\leq 2\lt(t)$ and $\ell_1'(g)\leq 2\tilde{\lt}(g)$ for any $g\in \SU(n)$ hold.

Now consider $\SO(2n+1)\subset \U(2n+1)$. An element in the maximal torus of $\SO(2n+1)$ then has the form $t=\diag(t_1, \ldots,t_n,1,t_n^{-1},\ldots ,t_1^{-1})$. The characters corresponding to fundamental roots are given by $\beta_i(t)=t_it_{i+1}^{-1}$ for $i=1\ldots n-1$ and $\beta_n(t)=t_n$. We want to proceed as in the case of $\SU(n)$ but have to take care of the last fundamental root.
Every root has to be estimated twice, since $\ell_1$ counts both $|1-t_i|$ and $|1-t_i^{-1}|$, and so we apply the estimate used for $\SU(n)$ to $2|1-zt_i(zt_{i+1})^{-1}|=|1-zt_i(zt_{i+1})^{-1}|+|1-zt_i^{-1}z^{-1}t_{i+1}|$ to obtain
\[2|1-zt_i(zt_{i+1})^{-1}| \leq |1-zt_i|+|1-zt_{i+1}|+ |1-zt_i^{-1}|+|1-zt_{i+1}^{-1}|.\]
Noting that $t_{n+1}=1$ and $t_i^{-1}=t_{2n+2-i}$, we see
\begin{align*}
\lt(t) &=\frac{1}{\pi n}\sum_{i=1}^nl(t_it_{i+1}^{-1})\leq \frac{1}{2n}\sum_{i=1}^n|1-zt_i(zt_{i+1})^{-1}|\\
&\leq \frac{1}{2n}\sum_{i=1}^n|1-zt_i|+|1-zt_i^{-1}|+\frac{1}{4n}(|1-z|+|1-z^{-1}|)\\
&= \frac{1}{2n}\sum_{i=1}^{2n+1}|1-zt_i|= \frac{2n+1}{n}\ell_1(zt).
\end{align*}
By taking the infimum over all $z\in Z(\U(2n+1))$,  $\lt(t)\leq \ell_1'(t)$ follows, and because this is independent of the ordering of the $t_i$ also $\tilde{\lambda}(g)\leq \ell_1'(g)$, where $g$ is arbitrary.
To reorder the diagonal entries of $\diag(t_1,\ldots, t_n,1,t_n^{-1},\ldots,t_1^{-1})$ by conjugation with a permutation matrix there is the possibility to permute $t_i$ and $t_i^{-1}$ and to permute the first $n$ entries, which entails corresponding permutation of the last $n$. Hence without loss of generality we can assume $l(t_1t_2^{-1})$ maximal among all $l(t_it_j^{\pm 1})$ and $l(t_it_{i+1}^{-1})\geq (t_it_j^{\pm 1})$ for $j>i$. Then
\begin{align*}
\inf_{z\in Z(\U(n))}  \sum_{i=1}^n l(t_iz)+l(t_i^{-1}z)+l(z)&\leq \inf_{z\in Z(\U(n))} \sum_{i=1}^{n-1} l(t_it_n)+l(t_i^{-1}t_n)+l(z)\\
&\leq 2\sum_{i=1}^{n-1} l(t_it_{i+1}^{-1})+3l(t_n)
\end{align*}
implies $\ell_1'(t)\leq 6\lt(t)$ and consequently $\ell_1'(g)\leq 6\tilde{\lt}(g)$ for any $g\in \SO(2n+1)$.

We continue right away with $\SO(2n)$, where the characters evaluate as $\beta_i(t)=t_it_{i+1}^{-1}$ if $i<n$ and $\beta_n(t)=t_{n-1}t_n$. The computations to obtain $\tilde{\lt}(g)\leq 2\ell_1'(g)$ are done as in the case of $\SO(2n+1)$ without the difficulties resulting from the different root and the odd dimension. We can permute entries of diagonal elements similarly to the case of $\SO(2n+1)$ with the restriction of performing only an even number of exchanges $t_i\mapsto t_i^{-1}$. But regardless of whether we have $t_n^{\pm 1}$ in the right place, the estimate works as in the case of $\SO(2n+1)$.

In $\Sp(2n)$ the characters are given by $\beta_i(t)=t_it_{i+1}^{-1}$ when $i<n$ and by $\beta_n(t)=t_n^2$. We have the same possibilities to permute $t_i$, $t_j$ and $t_i$, $t_i^{-1}$ as in $\SO(2n+1)$ and no additional $1$ on the diagonal. So everything works fine.

At last let $G$ be one of the exceptional groups. Let the standard representation $\delta'$ embed $G$ into $\U(2n)$, and $\omega_1,\ldots,\omega_n$ be the weights of $\delta$. Then $\delta'$ has the weights $\omega_1,\ldots \omega_n$ and $\omega_{n+i}:=\omega_i^{-1}$ for $i=1\ldots n$, and the diagonal elements in $U(2n)$ coming from $G$ take the form $\diag(\omega_1(t),\ldots ,\omega_{2n}(t))$ for $t\in T$. Since the root lattice is contained in the weight lattice, every fundamental root is a linear combination of weights with integer coefficients, $\beta_i=\sum_{j=1}^nm_{ij}\omega_j$ say.
Note that for an element $y\in S^1$ we have $|1-y|+|1-y^{-1}|\leq 2(|1-zy|+|1-zy^{-1}|)$ for any $z\in S^1$ if $l(y)\leq \pi/2$. 
Let $m_i:=\sum_{j=1}^n|m_{ij}|$. Since we assumed $\tilde{\lt}(t)\leq \frac{1}{2r}$, $l(\beta_i(t))\leq \pi/2$ for every $i$ and we estimate
\begin{align*}
|1-\beta_i(t)|&= \tfrac{1}{2}(|1-\beta_i(t)|+|1-\beta_i(t)^{-1}|)\\
&\leq |1-z^{m_i}\beta_i(t)|+|1-z^{m_i}\beta_i(t)^{-1}|\\
&=\left|1-\prod_{j=i}^nz^{|m_{ij}|}\omega_j(t)^{m_{ij}}\right|+\left|1-\prod_{j=i}^nz^{|m_{ij}|}\omega_j(t)^{-m_{ij}}\right|\\
&\leq \sum_{j=1}^n|m_{ij}||1-z\omega_j(t)^{\sgn m_{ij}}|+\sum_{j=1}^n|m_{ij}||1-z\omega_j(t)^{-\sgn m_{ij}}|\\
&= \sum_{j=1}^n|m_{ij}|(|1-z\omega_j(t)^{m_{ij}}|+ |1-z\omega_j(t)^{-m_{ij}}|).
\end{align*}
By summing over all $i=1\ldots r$ we obtain
\[\sum_{i=1}^r|1-\beta_i(t)|\leq \sum_{i=1}^r\sum_{i=1}^n|m_{ij}||1-\omega_j(t)^{\pm 1}|\leq \sum_{j=1}^{2n}M|1-z\omega_j(t)|,\]
where $M:=\max_{j=1\ldots n}\sum_{i=1}^r|m_{ij}|$. By appropriate scaling of the two sums and taking the infimum over all $z$ we arrive at $\lt(t)\leq M\ell_1'(t)$ and also $\tilde{\lt}(g)\leq M\ell_1'(g)$ for all $g\in G$.

By looking up the tables in \cite{knapp02}, Appendix~C, we find that the highest weight $\omega$ for $\delta$ is a linear combination of simple roots with integer coefficients. All other weights differ from $\omega$ by an element of the root lattice and therefore every weight is a linear combination of the fundamental roots with integer coefficients. We write $\omega_j=\sum_{i=1}^r n_{ji}\beta_i$, where $n_{ji}\in\mathbb{Z}$. Then, using the special number $z=1$, similarly to the above calculation
\[\sum_{j=1}^n|1-z\omega_j(t)^{\pm 1}|\leq \sum_{j=1}^n|1-\omega_j(t)^{\pm 1}|\leq \tfrac{1}{2}\sum_{i=1}^rN|1-\beta_i(t)^{\pm 1}|\]
follows, where $N:=\max_{i=1\ldots r}\sum_{j=1}^n|n_{ji}|$. After rescaling $\ell_1'(t)\leq 2N\tilde{\lt}(t)$ follows. Setting $L:=\max(M,2N)$ finishes the proof.
\end{proof}

For the next theorem we also write $\tilde{\lt}$ for the function obtained as a pointwise ultralimit of the functions $\tilde{\lt}$ in ultraproducts of Lie groups.

\begin{theorem}
Let $G_n$ be compact connected quasisimple Lie groups. If $\bs{g}\in\bs{G}:=\prod_{\mathfrak{u}}G_n$ satisfies $\lt(\bs{g})>0$, then $N(\bs{g})=\bs{G}$. The set $\bs{N}$ of all $\bs{g}$ such that $\tilde{\lt}(\bs{g})=0$ is a normal subgroup and $\bs{G}/\bs{N}$ is simple.
\end{theorem}

\begin{proof} We can assume $g_n\in T_n$, where $T_n$ is a maximal torus of $G_n$. Then the first part of the theorem follows already from Lemma~5.19 in \cite{nikolovsegal12}. For groups of bounded rank we can alternatively use Theorem~\ref{thm:BoundedGenerationOfSmallElementsInLieGroupsOfSmallRank}.

By Lemma~\ref{lem:LOneLengthOnCompactLieGroups} $\tilde{\lt}(\bs{g})=0$ is equivalent to $\ell_1'(\bs{g})=0$. Because $\ell_1'$ is a pseudo length function, $\bs{N}$ is a normal subgroup. From the first part of the theorem we deduce that $\bs{G}/\bs{N}$ is simple.
\end{proof}

We define $\bs{g}\preceq \bs{h}$ for $\bs{g},\bs{h}\in \bs{G}\setminus \{1\}$ as in Paragraph~\ref{ssec:NormalSubgroupsOfUltraproductsOfFiniteSimpleGroups}, except that we use $\ell_1'$ as our length of choice. Then Lemma~\ref{lem:OrderAndLengthConnection} immediately implies $\bs{g}\preceq \bs{h}$ whenever $\bs{g}\in N(\bs{h})$.

\begin{lemma}
Let $\bs{G}$ be an ultraproduct of compact connected simple Lie groups $G_n$ of bounded rank and assume $\bs{g}\preceq \bs{h}$ for non-trivial elements $\bs{g}$ and $\bs{h}$ in $\bs{G}$. Then $\bs{g}\in N(\bs{h})$. 
\end{lemma}

\begin{proof}
The hypothesis assures $\lt(g_n)\leq m\lt(h_n)$ for almost all $n$ and a suitable constant $m$. Following Theorem~\ref{thm:BoundedGenerationOfSmallElementsInLieGroupsOfSmallRank} we immediately obtain $g_n\in C(h_n^{\pm 1})^{4mr^2}$, where $r$ is the bound on the rank of the groups $G_n$. Hence \[\bs{g}\in C(\bs{h}^{\pm 1})^{4mr^2}\subset N(\bs{h})\]
\end{proof}

We are now ready to prove the analogue of Theorem~\ref{thm:NormalSubgroupsOfUltraproductsOfFiniteSimpleGroups} for Lie groups of bounded rank.

\begin{theorem}\label{thm:NormalSubgroupsOfUltraproductsOfLieGroups}
Let $G_n$ be compact connected simple Lie groups of bounded rank. Then the set $\mathfrak{N}$ of normal subgroups of $\bs{G}:=\prod_{\mathfrak{u}}G_n$ is linearly ordered by inclusion. 
\end{theorem}

\begin{proof}
 Exactly as in the proof of Theorem~\ref{thm:NormalSubgroupsOfUltraproductsOfFiniteSimpleGroups} we show that the set $\mathfrak{N}_0$ of normal closures of elements of $\bs{G}$ is order isomorphic to a subset of $\bs{K}/\equiv$. This is the quotient of $\bs{K}:=\prod_{\mathfrak{u}}[0,1]$ by the equivalence relation $\equiv$, which defines $\bs{a}$ and $\bs{b}$ equivalent if
\[0<\lim_{\mathfrak{u}}\frac{a_n}{b_n}<\infty.\]
Because a maximal torus $T$ in a Lie group of rank $r$ is isomorphic to the standard torus $(S^1)^r$, it is clear that for any prescribed $a$ in $[0,1]$ there is an element in $T$ with length $a$. Hence $\mathfrak{N}_0$ is isomorphic to $\bs{K}/\equiv$.
 Now an application of Lemma~\ref{lem:OrderOfNormalSubgroupsAndOrderOfNormalClosures} shows that also $\mathfrak{N}$ is linearly ordered.
\end{proof}

Unfortunately, unlike for finite simple groups, the theorem turns out to be false if there is no bound on the rank. We illustrate this fact as follows.

Let $G_n:=\SU(2n+1)$. We consider elements
\begin{align*}
g_n&=\diag(e^{\imag 2\pi (n-1)/n}, e^{\imag\pi/n^2},e^{\imag \pi /n^2}, \ldots ,e^{\imag \pi /n^2},1,\ldots,1),\\
 h_n&=\diag(-1,-1,1,\ldots,1)
\end{align*}
in the maximal torus $T$ of diagonal matrices, where $n$ entries of $g_n$ equal $1$.
To make the counterexample meaningful we have to pass to $\PSU(2n+1)$, or equivalently use pseudo length functions that vanish on $Z=Z(\SU(2n+1))$. 
If we assume that $\bs{g}\in N(\bs{h})$, then Lemma~\ref{lem:OrderAndLengthConnection} implies the existence of a constant $m$ such that $\inf_{z\in Z}\lr(zg_n)\leq m\inf_{z\in Z}\lr(zh_n)$ for infinitely many $n$. But obviously the left hand side converges to $1/2$ and the right hand side equals $\frac{2m}{2n+1}$, which does not fit together well. On the other hand $\bs{h}\in N(\bs{g})$ would imply
\[\tilde{\lt}(h_n)\leq 2\ell_1'(h_n)\leq 2m\ell_1'(g_n)\leq 4m\tilde{\lt}(g_n),\]
 for some (other) constant $m$ and infinitely many $n$. After ordering the entries of $g_n$ and $h_n$ appropriately we evaluate $\tilde{\lt}(g_n)\leq \frac{4}{n(2n+1)}$ and $\tilde{\lt}(h_n)=\frac{4}{2n+1}$ to obtain a contradiction once again, and must conclude that neither $N(\bs{g})\subset N(\bs{h})$ nor $N(\bs{h})\subset N(\bs{g})$ holds. Therefore the normal subgroups in the ultraproduct of (projective) unitary groups cannot be linearly ordered.

Despite this setback we try to see how far we can get. Let $g$ be an element a compact connected quasisimple Lie group of rank $n$ with maximal torus $T$. For the rest of the section we will call $t\in C(g)\cap T$ \deph{optimal} if the following holds: For all $s\in C(g)\cap T$ we have $|1-\beta_1(t)|\geq |1-\beta_1(s)|$, and for all $k=1\ldots n-1$ the equation $\sum_{i=1}^k |1-\beta_i(t)| = \sum_{i=1}^k|1-\beta_i(s)|$ implies $|1-\beta_{k+1}(t)|\geq |1-\beta_{k+1}(s)|$. We define a function $F_g:\mathbb{N}\rightarrow [0,1]$ by 
\[F_g(i):=\left\{\begin{array}{l l }
		    \tfrac{1}{2}|1-\beta_{\sigma(i)}(t)|, & i\in\mathbb[n],\\
		    0, & i> n,
                 \end{array}\right.
\]
where $t$ is optimal and $\sigma$ is a permutation of $[n]$ such that $F_g(i)\geq F_g(i+1)$ results for all $i\geq 1$. Note that there is always an optimal $t$ and for different optimal elements $s$ and $t$ the differences $|1-\beta_i(t)|$ and $|1-\beta_i(s)|$ are the same. Hence $F_g$ is well defined.

Let the sequences of functions $(F_n)$ and $(H_n)$ be representatives of elements $\bs{F}$ and $\bs{H}$, respectively, in the ultraproduct
\[\bs{M}:=\prod_{\mathfrak{u}}\mathscr{F}_n(\mathbb{N},[0,1]),\]
where $\mathscr{F}_n(\mathbb{N},[0,1])$ is the set of decreasing functions $\mathbb{N}\rightarrow [0,1]$ with support contained in $[n]$.
We let $\bs{F}\preceq \bs{H}$ if and only if there are constants $c$ and $k\in\mathbb{N}$ such that for $\mathfrak{u}$-almost all $n$
\[F_n(ki+1)\leq cH_n(i+1),\]
whenever $i\geq 0$. It is clear that this defines a quasiorder on the space $\bs{M}$.
We let $\bs{F}\equiv \bs{H}$ if $\bs{F}\preceq \bs{H}$ and $\bs{H}\preceq \bs{F}$ to obtain the the quotient space $\bs{M}/\equiv$ with the induced ordering.

If $\bs{g}\in \bs{G}\setminus\{1\}$ we define $F_{\bs{g}}$ as the element in $\bs{M}$ associated with $(F_{g_n})_n$. With these two notions at hand let $\bs{g}\preceq \bs{h}$ be equivalent to $F_{\bs{g}}\preceq F_{\bs{h}}$.

\begin{lemma}\label{lem:SingularValuesEstimate}
 Let $g$ and $h$ be elements in a classical compact connected quasisimple Lie group $G$. Then, if $i,j\geq 0$,
\[F_{gh}(6i+6j+1)\leq 2F_g(i+1)+2F_h(j+1).\]
\end{lemma}

\begin{proof}
Consider the standard embedding of $G$ in $\U(n)$. The singular values $s_i(1-g)$ of $1-g$ are defined as 
\[s_i(1-g)=\sqrt{\lambda_i((1-g)^*(1-g))},\]
where  $(1-g)^*$ is the adjoint operator of $(1-g)$, $\lambda_i=\lambda_i((1-g)^*(1-g))$ is some eigenvalue of $(1-g)^*(1-g)$, and we assume the $s_i(1-g)$ in decreasing order. Then $s_i(1-g)=|1-\lambda_i|$.
We define
\[\underline{s}_i(g):=\inf_{z\in Z(\U(n))}\tfrac{1}{2} s_i(1-zg).\]
Observe that since $s_{i+1}(1-zg)\leq s_i(1-zg)$ holds for all $z$, also $\underline{s}_{i+1}(g)\leq \underline{s}_i(g)$ is true for $i=1\ldots n$.

First study $G=\SU(n)$ and let $t\in C(g)\cap T$ be optimal. Then $F_g(i)=\frac{1}{2}|t_{\sigma(i)}-t_{\sigma(i)+1}|$. Let $z\in Z(\U(n))$ and $\tau$ a permutation such that $|z-t_{\tau(1)}|\geq |z-t_{\tau(2)}|\geq \ldots \geq |z-t_{\tau(n)}|$.
If we assume the existence of $i$ such that $|t_{\sigma(2i+1)}-t_{\sigma(2i+1)+1}|>2|z-t_{\tau(i+1)}|$, then for all $k=1\ldots 2i+1$ the estimate $2|z-t_{\tau(i+1)}|<|z-t_{\sigma(k)}|+|z-t_{\sigma(k)+1}|$ follows. Hence $|z-t_{\sigma(k)}|>|z-t_{\tau(i+1)}|$ or $|z-t_{\sigma(k)+1}|>|z-t_{\tau(i+1)}|$ and $\sigma(k)\in\{\tau(1),\ldots,\tau(i)\}$ or $\sigma(k)+1\in\{\tau(1),\ldots,\tau(i)\}$. It might happen that $\sigma(k)=\sigma(l)+1$ for some $1\leq k,l\leq i$, but not more than $i$ times. Thus $\{\tau(1),\ldots, \tau(i)\}$ contains at least $i+1$ elements, a contradiction. Therefore $F_g(2i+1)\leq s_{i+1}(1-zg)$ holds for all $i$, independently of $z$.

If $i,j\geq 0$ and $i+j+1\leq n$ we have for central elements $x$, $y$ in $\U(n)$ by the Ky Fan singular value inequality 
\begin{align*}
s_{i+j+1}(1-xygh)&=s_{i+j+1}((1-xg)yh+(1-yh))\\
&\leq s_{i+1}((1-xg)yh)+s_{j+1}(1-yh)\\
&=s_{i+1}(1-xg)+s_{j+1}(1-yh).
\end{align*}
(Confer \cite{fan51}.) Combining this estimate with the previous one we obtain
\[F_{gh}(2i+2j+1)\leq s_{i+1}(1-xg)+s_{j+1}(1-yh).\]
 Taking the infimum over all $x,y\in Z(\U(n))$ on both sides yields
\[F_{gh}(2i+2j+1)\leq 2\underline{s}_{i+1}(g)+2\underline{s}_{j+1}(h).\]
For the other classical groups the proof is similar and we will point out where slight changes have to be made. Consider the case $G=\SO(2n+1)\subset \U(2n+1)$. Then the contradiction at the beginning of the proof above can be produced by assuming $|1-\beta_{\sigma(3i+1)}(t)|>2|z-t_{\tau(i+1)}|$ (when $3i+1\leq n$), where $\tau$ satisfies $|z-t_{\tau(1)}|\geq |z-t_{\tau(2)}|\geq \ldots \geq |z-t_{\tau(2n+1)}|$. Since we have $\beta_n(t)=t_n$, for $i_0$ such that $\sigma(i_0)=n$ we must distinguish two cases. If $i_0=1$ the maximality of $|z-t_{\tau(1)}|$ is contradicted by $2|z-t_{\tau(1)}|<|z-t_{\sigma(1)}|+|z-1|$. (Note that $|z-1|$ is a singular value of $z-g$.) For other $i_0$ the proof works as above, and we use the factor $3$ to be able to ignore $i_0$ when deducing the contradiction. In $\SO(2n)$ or $\Sp(2n)$ we proceed in analogy with the procedure for $\SO(2n+1)$.
Hence for these groups
\[F_{gh}(3i+3j+1)\leq 2\underline{s}_{i+1}(g)+2\underline{s}_{j+1}(h).\]
follows.

Consider $\SU(n+1)$ and let $t$ be optimal. Because we can conjugate with arbitrary permutation matrices of determinant $1$ in $\SU(n+1)$, by definition of optimality 
$|t_i-t_{i+1}|\geq |t_i-t_j|$ for all $i=1\ldots n-2$, $j\geq i+1$. Let $\tau$ be a permutation such that $|t_{n-1}-t_{\tau(1)}|\geq |t_{n-1}-t_{\tau(2)}|\geq \ldots \geq |t_{n-1}-t_{\tau(n)}|$. Then
 \begin{align*}
2\underline{s}_i(g)&\leq \inf_{z\in Z(\U(n))} s_i(z-t)\\
&\leq s_i(t_{n-1}-t)\\
&= |t_{n-1}-t_{\tau(i)}|\\
&\leq |t_{\tau(i)}-t_{\tau(i)+1}|\\
& \leq 2F_g(\sigma^{-1}\tau(i))  
 \end{align*}
 if $\tau(i)\leq n-2$. If $\tau(i)=n-1$ we have $\underline{s}_i(g)\leq 0$ and if $\tau(i)=n$, then $2\underline{s}_i(g)\leq |t_{n-1}-t_n|$. Because $F_g$ is decreasing by definition and we can estimate each $\underline{s}_i(g)$ from above with a unique  value of $F_g$, even
\[\underline{s}_i(g)\leq F_g(i)\]
follows for all $i=1\ldots n$.

Now let $G=\SO(2n+1)$ and $t$ optimal. Then $|t_i-t_{i+1}|\geq |t_i-t_j^{\pm 1}|$ for all $i=1\ldots n-1$, $i+1\leq j\leq n$. Let $\tau$ be a permutation such that $|t_n-t_{\tau(1)}|\geq |t_n-t_{\tau(2)}|\geq \ldots \geq |t_n-t_{\tau(n)}|$. We have to take into account that $|t_n-t_i|$ and $|t_n-t_i^{-1}|$ might be of comparable size. Therefore, introducing the factor $2$, for $i\geq 0$
 \[2\underline{s}_{2i+1}(g) \leq s_{2i+1}(t_n-t)\leq |t_n-t_{\tau(i+1)}|\leq |t_{\tau(i+1)}-t_{\tau(i+1)+1}| \leq 2F_g(\sigma^{-1}\tau(i+1))\]
 if $\tau(i+1)\neq n$, and $\underline{s}_{2i+1}(g)\leq 0$ otherwise. Hence
$\underline{s}_{2i+1}(g)\leq F_g(i)$ for all $i=1\ldots n$.

In $G=\SO(2n)$ we have  $|t_i-t_{i+1}|\geq |t_i-t_j^{\pm 1}|$ for all $i=1\ldots n-2$, $i+1\leq j\leq n$. Let $\tau$ be a permutation as above corresponding to $t_{n-1}$. Then we proceed as above to obtain $\underline{s}_{2i+1}(g)\leq F_g(i)$ for all $i=1\ldots n$. The case of $\Sp(2n)$ is similar.

Combining the different estimates finishes the proof.
\end{proof}

\begin{proposition}\label{prop:SingularValueOrderingAndLieGroupElementsOne}
 Let $\bs{g}$ and $\bs{h}$, both not equal to $1$, be elements in an ultraproduct of compact connected simple Lie groups $G_n$ such that $\bs{g}\in N(\bs{h})$.
Then $\bs{g}\preceq \bs{h}$.
\end{proposition}

\begin{proof}
Let $G_n$ have rank $m_n$ and consider the interesting case $m_n\rightarrow_{\mathfrak{u}}\infty$. We assume that $\bs{g}$ is a product of $k$ conjugates of $\bs{h}^{\pm1}$. This implies that $g_n\in G_n$ is a product of not more than $6k$ conjugates of $h_n^{\pm1}$ for $\mathfrak{u}$-almost all $n$. By conjugating we can assume $g_n$ and $h_n$ in a maximal torus of $G_n$. We only have to take care of $m_n$ sufficiently larger than $6k$. Imagine $G_n$ embedded in a unitary group by the standard representation, in order to use Lemma~\ref{lem:SingularValuesEstimate}. In a group of such a large rank now for $i\geq 0$ 
\[F_g(6ki+1)\leq 2^k6kF_h(i+1)\]
holds, because $F_h$ is invariant under conjugation of $h$ with unitaries.
\end{proof}

A graph $X$ has \deph{coloring number} $\chi(X)=k$ if there is a coloring of the vertices with $k$ colors such that no two vertices of the same color are joined by an edge and $k$ is minimal with this property.

Let $X$ be a graph with a partition of the vertices into subsets of size $k$. Then a \deph{strong} $k$\deph{-coloring} of $X$ is a coloring such that every color appears in each partition exactly once. (If the number of vertices is not divisible by $k$ we add isolated vertices as needed.) Then the \deph{strong coloring number} $s\chi(X)$ of $X$ is the least $k$ such that for all partitions of the vertices into subsets of size $k$, $X$ admits a strong $k$-coloring. (Confer also \cite{alon92}.)

\begin{lemma}\label{lem:PartitioningFiniteOrderedSets}
 There is a natural number $s\geq 3$ such that the following holds. Let $\sigma$ be a permutation of the numbers $[n]$, where $n$ is divisible by $s$. Then one can partition $(\sigma(1),\sigma(2),\ldots, \sigma(n))$ into $s$ vectors $v_i:=(a_{i,1},\ldots,a_{i,n/s})$ such that $|a_{i,j}-a_{i,k}|\neq 1$ for all $i=1\ldots s$ and $j,k=1\ldots n/s$, and $a_{i,j}=\sigma(k)$ implies $|sj-k|\leq s-1$.
\end{lemma}

\begin{proof}
 We reformulate the problem in graph theoretical terms. Consider the vector $(1,2,\ldots,n)$ as a graph, where $i,j$ are connected if $|i-j|\in \{1,n-1\}$, i.e.\ the cycle $C_n$. We assign a vertex $i$ of this graph the label $\lfloor \frac{\sigma^{-1}(i)+s-1}{s}\rfloor$. Thus we use $n/s$ different labels, each one occuring exactly $s$ times. If $s\geq s\chi(C_n)$, then there is a proper coloring of $C_n$ such that no two vertices with the same label have the same color. Now let $v_i$ be the vector of vertices of color $i$, in the ordering prescribed by the labels. Then it follows immediately that no two consecutive numbers appear in the same $v_i$. If $a_{i,j}=\sigma(k)$, then $j=\lfloor \frac{k+s-1}{s}\rfloor$ and the difference $|sj-k|$ is strictly less than $s$. 

Since it is known that the strong coloring number of $C_n$ can be bounded independently of $n$, the claim follows.
\end{proof}

Note that the constant $s\chi(C_n)$ in the previous lemma can be made explicit. Alon in \cite{alon92} mentions the bound of $s\chi(C_n)\leq 4$ (for $n$ divisible by $4$), credited to de la Vega, Fellows and himself. The usual proofs invoke probabilistic methods such as the Lov\'asz local lemma. Fleischner and Stiebnitz proved $s\chi(C_n)=3$ and there is an elementary proof, presented by Sachs in \cite{sachs93}.

\begin{lemma}\label{lem:BoundedGenerationOfSmallElementsInLieGroupsOfLargeRank}
 Let $G$ be a compact connected simple Lie group of rank $r>20k$ for a natural number $k$. Assume $g$ and $h$ are non-trivial elements in the maximal torus $T$ satisfying $F_g(ki+1)\leq mF_h(i+1)$ if $i\geq 0$, where $m\in\mathbb{N}$ is an even integer. Then
\[g\in (h^G\cup h^{-G})^{140km+4m}.\]
\end{lemma}

\begin{proof}
Considering the rank requirements, $G$ is a classical group of type ${\rm A}_r$, ${\rm B}_r$, ${\rm C}_r$ or ${\rm D}_r$. Then without loss of generality the roots $\beta_i$, $i=1\ldots r-1$, form a root system of type ${\rm A}_{r-1}$ and the root $\beta_r$ possibly has a different length. Roots $\beta_j$ and $\beta_i$ are orthogonal, whenever $|i-j|\geq 2$ and we will say that $i$ is orthogonal to $j$ in that case.

We can assume without loss of generality that $g$ and $h$ are optimal. Let $K:= 5k$. With $N$ the largest natural number divisible by $3$ such that $NK\leq r-K-1$, we define $N$-tuples $A_l:= (l,K+l,2K+l,\ldots, NK+l)$ for $l=1\ldots K$ and $A_0:= (1,2,\ldots, N)$. We choose the permutation $\sigma$ implicitly by writing $F_g(i)=\tfrac{1}{2}|1-\beta_{\sigma(i)}(g)|$ as above. Likewise we have $\tau$ corresponding to $h$. Both permutations act coordinatewise on $N$-tuples. If we choose $i\geq 0$, then
\[\tfrac{1}{2}|1-\beta_{\sigma(Ki+l)}(g)|=F_g(Ki+l)\leq mF_h(i+1)=\tfrac{1}{2}m|1-\beta_{\tau(i+1)}(h)|.\]
If $l\in\{1,\ldots,K\}$ we hence obtain
\[l(\beta_{\sigma(Ki+l)}(g))\leq 2ml(\beta_{\tau(i+1)}(h)).\]

Without loss of generality, we can assume the worst case that $\tau(A_0)$ contains $N$ consecutive numbers. Then by Lemma~\ref{lem:PartitioningFiniteOrderedSets} and the subsequent remarks there is a partition of $\tau(A_0)$ into tuples $B_1$, $B_2$ and $B_3$ with the same number of elements such that the entries in $B_i$ are pairwise orthogonal for $i=1,2,3$. In the same way we obtain $C_{i,l}$, $i=1,2,3$, from the sets $\sigma(A_l)$, $l=1\ldots K$. By Lemma~5.21 in \cite{nikolovsegal12} (here we use orthogonality) there are elements $w_{i,l}$ in the Weyl group of $G$ that map the vectors of fundamental roots corresponding to $B_i$ to the vectors of fundamental roots corresponding to $C_{i,l}$ for all $i=1,2,3$, $l=1\ldots k$. We will apply Lemma~\ref{lem:BoundedGenerationOfTorusParts} simultaneously to all $g_i$ for $i\in C_{j,l}$, but have to check first if we will end up with good enough constants. Lemma~\ref{lem:PartitioningFiniteOrderedSets} with $s=3$ guarantees that indices are at a distance of at most $2$ from their optimal position. In the worst case we have to compare $l(\beta_{\sigma(K(i-2)+l)}(g))$ with $l(\beta_{\tau(i+3)}(h))$. Since under this assumption $i\geq 3$,
\begin{align*}
 K(i-2)+l &\geq 5k(i-2)+1 = ki+1+4ki-10k\\
& \geq (ki+1)+12k -10k = k(i+2)+1.
\end{align*}
This implies 
\[l(\beta_{\sigma(K(i-2)+l)}(g))\leq l(\beta_{\sigma(k(i+2)+1)}(g))\leq 2ml(\beta_{\tau(i+3)}(h)).\]
 For other possibly dislocated indices this kind of estimate works as well. After the abovementioned application of Lemma~\ref{lem:BoundedGenerationOfTorusParts} we know 
\[\prod_{i\in C_{j,l}}g_i \in (h^G\cup h^{-G})^{2\cdot 2m}.\]
 When we reconstruct most of $g$ in this way, we arrive at
\[\prod_{i\in \bigcup C_{j,l}}g_i \in (h^G\cup h^{-G})^{4m\cdot 15k},\]
because we had to treat $3\cdot K=15 k$ sets $C_{j,l}$.
What remains are the indices left out in the above procedure. The number of these is $r-NK\leq 4K$ by the choice of $N$. If $i\leq r-1$, using $h_{\tau(1)}$, we can generate the $g_i$ separately in $2\cdot 2m$ steps as in Lemma~\ref{lem:BoundedGenerationOfTorusParts}. The last root $\beta_r$ possibly requires the second argument in the proof of Lemma~\ref{lem:BoundedGenerationOfTorusParts}, which results in adding $4\cdot 2m$. Hence generating the missing parts of $g$ can be done in $(4K-1)\cdot 4m+ 8m= 4(20k+1)m$ steps.

All in all we end up with
\[g\in (h^G\cup h^{-G})^{140km+4m}\]
as claimed.
\end{proof}

\begin{theorem}
 Let $\bs{g}$ and $\bs{h}$ be elements in the ultraproduct $\bs{G}$ of compact connected simple Lie groups of unbounded rank. Then $\bs{g}\preceq \bs{h}$ is equivalent to $\bs{g}\in N(\bs{h})$. 
\end{theorem}

\begin{proof}\label{thm:SingularValueOrderingAndLieGroupElementsTwo}
The first implication was already proved in Proposition~\ref{prop:SingularValueOrderingAndLieGroupElementsOne}. The proof of the second is an application of Lemma~\ref{lem:BoundedGenerationOfSmallElementsInLieGroupsOfLargeRank}, analogous to the proofs of Theorem~\ref{thm:NormalSubgroupsOfUltraproductsOfFiniteSimpleGroups} or Theorem~\ref{thm:NormalSubgroupsOfUltraproductsOfLieGroups}.
\end{proof}

Up to now it is clear that the set of normal closures of elements in $\bs{G}$ is order isomorphic to $\bs{M}/\equiv$. What remains to be clarified is the influence of this ordering on the ordering of normal subgroups. 

\subsection{The lattice of normal subgroups}
We are interested in the lattice of normal subgroups of groups $G$. The lattice operations are $N\wedge M=N\cap M$ and $N\vee M=NM$, the normal subgroup generated by $N$ and $M$, for any choice of normal subgroups in $G$. It is well known that the lattice of normal subgroups of any group is modular, that is for normal subgroups $L,M,N$ the modular law
\[((L\wedge N)\vee M)\wedge N=(L\wedge N)\vee(M\wedge N)\]
holds.

\begin{lemma}\label{lem:MeetAndJoinOfNormalClosures}
 Let $\bs{G}$ be an ultraproduct of compact connected simple Lie groups and $\bs{g}$, $\bs{h}$ in a maximal torus $\bs{T}$ of $\bs{G}$. Then there are $\bs{a}, \bs{b}\in \bs{T}$ such that $N(\bs{g})\wedge N(\bs{h})=N(\bs{a})$ and $N(\bs{g})\vee N(\bs{h})=N(\bs{b})$.
\end{lemma}

\begin{proof}
We define functions $\bs{A}:=\min(F_{\bs{g}}, F_{\bs{h}})$ and $\bs{B}:=\max(F_{\bs{g}}, F_{\bs{h}})$. The plan is to show that there are actually elements $\bs{a}$ and $\bs{b}$ such that $\bs{A}=F_{\bs{a}}$ and $\bs{B}=F_{\bs{b}}$. For some $n$ consider the functions $A_n:=\min(F_{g_n},F_{h_n})$, $B_n:=\min(k\mapsto 1,\max(F_{g_n},F_{h_n}))$. Let $T_n$ be a maximal torus in the group $G_n$ of rank $r$, where we can assume $g_n,h_n\in T_n$.
Because $T_n$ is isomorphic to  $(S^1)^r$ we find elements $a_n$ and $b_n$ in $T_n$ such that $F_{a_n}=A_n$ and $F_{b_n}=B_n$. This yields $\bs{a}$ and $\bs{b}$ as claimed.
\end{proof}

\begin{proposition}\label{prop:DistributiveLatticeOfNormalClosures}
 Let $\bs{G}$ be an ultraproduct of compact connected simple Lie groups. Then the set $\mathfrak{N}_0$ of normal closures of elements in $\bs{G}\setminus\{1\}$ is a distributive lattice.
\end{proposition}

\begin{proof}
We already know that $\mathfrak{N}_0$ is order isomorphic to $\bs{M}/\equiv$. It is clear that the latter is a distributive lattice with meet and join induced by the operations $\min$ and $\max$ applied to functions. Lemma~\ref{lem:MeetAndJoinOfNormalClosures} shows that the corresponding operations in $\mathfrak{N}_0$ produce normal closures again.
\end{proof}

\begin{lemma}\label{lem:DistributiveLatticeOfNormalSubgroups}
Let $G$ be a group. If the set of normal closures of elements in $G$ is a distributive lattice, then the lattice of normal subgroups is distributive, too.
\end{lemma}

\begin{proof}
Let $L,M,N$ be any normal subgroups in $G$. We have to show that
\[(L\vee M) \wedge N= (L\wedge N)\vee (M\wedge N)\]
holds. Here the inclusion of the right hand side in the left hand side is true in general. Moreover  by assumption the whole equation holds for normal closures of elements in $G$.
Consider $x\in (L\vee M)\wedge N$. Then $x\in L\vee M$ and $x\in N$ because the meet operation is intersection of sets. Because the normal closure of $L$ and $M$ is the normal subgroup $LM$, there are $a \in L$ and $b\in  M$ such that $x$ equals the product $ab$. This means that $x\in N(a)\vee N(b)$.
 We also observe $N(x)\subset N$ to obtain
\begin{align*}
 x &\in (N(a)\vee N(b))\wedge N(x)\\
& =(N(a)\wedge  N(x))\vee (N(b)\wedge N(x))\\
&\subset (L\wedge N)\vee (M\wedge N).
\end{align*}
Thus the claim follows.
\end{proof}

The observations made in Proposition~\ref{prop:DistributiveLatticeOfNormalClosures} and Lemma~\ref{lem:MeetAndJoinOfNormalClosures} suffice to prove the following result.

\begin{theorem}\label{thm:DistributiveLatticeOfNormalSubgroupsOfUltraproducts}
 If $\bs{G}$ is an ultraproduct of compact connected simple Lie groups, then the lattice of normal subgroups of $\bs{G}$ is distributive.
\end{theorem}

\section{Conclusion}
We considered ultraproducts of finite simple groups and compact connected simple Lie groups. As a consequence of the Peter-Weyl Theorem, any compact simple group belongs to one of the two categories. We have to deal with the subcases of groups of bounded and unbounded rank, respectively, because the two behave differently as shown above. If we have an ultraproduct $\bs{G}$ of compact simple groups the ultrafilter selects one kind of groups among the four listed possibilities, which determine the properties of $\bs{G}$. We will say that $\bs{G}$ is of \deph{bounded finite type}, \deph{unbounded finite type}, \deph{bounded Lie type} or \deph{unbounded Lie type} if $\bs{G}$ is essentially an ultraproduct of finite simple groups of bounded or unbounded rank or Lie groups of bounded or unbounded rank, respectively.

Recall the situation in the case of finite simple groups. We defined $\bs{g}\preceq \bs{h}$ if 
\[\lim_{\mathfrak{u}}\frac{\ell(g_n)}{\ell(h_n)}<\infty,\]
where $\ell$ was one of the length functions $\lr$ and $\lj$.
For $g\neq 1$ in a finite simple group of rank $n$ define
\[F_g(k):= \left\{\begin{array}{ l l}
            0 & \text{otherwise},\\
	    1 & \text{if } k\leq n\ell(g).
           \end{array}\right.
\]
Then it is an elementary observation that $F_{\bs{g}}\preceq F_{\bs{h}}$, if and only if $\bs{g}\preceq \bs{h}$ for non-trivial $\bs{g},\bs{h}\in \bs{G}$.
Using this last remark we can summarize our results in the following theorem.

\begin{theorem}[Main Theorem]\label{thm:MainTheorem}
Let $\bs{G}$ be an ultraproduct of non-abelian compact simple  groups $G_n$. Let $\bs{M}$ be the ultraproduct of sequences of decreasing functions $F_n:\mathbb{N}\rightarrow [0,1]$ with support of size less or equal to the rank of $G_n$. Define $\bs{F}\preceq \bs{H}$ if there are constants $c,k$ such that $F_n(ki+1)\leq cH_n(i+1)$  for all $i\geq 0$ $\mathfrak{u}$-almost everywhere, and $\bs{F}\equiv \bs{H}$ if $\bs{F}\preceq \bs{H}$ as well as $\bs{H}\preceq \bs{F}$. 
\begin{enumerate}
 \item If $\bs{G}$ is of unbounded Lie type, then the set of normal closures $\mathfrak{N}_0$ of elements in $\bs{G}\setminus\{1\}$ is a lattice isomorphic to the distributive lattice $\bs{M}/\equiv$. The lattice $\mathfrak{N}$ of normal subgroups of $\bs{G}$ is distributive. 
\item If $\bs{G}$ is of bounded Lie type, then $\mathfrak{N}_0$ is isomorphic to the linearly ordered sublattice of $\bs{M}/\equiv$ induced by the functions of bounded support and $\mathfrak{N}$ is linearly ordered.
\item If $\bs{G}$ is of unbounded finite type, then $\mathfrak{N}_0$ is isomorphic to the linearly ordered sublattice of $\bs{M}/\equiv$ induced by the functions $F:\mathbb{N}\rightarrow \{0,1\}$. Again, $\mathfrak{N}$ is linearly ordered.
\item If $\bs{G}$ is of bounded finite type, then $\bs{G}$ is simple and $\mathfrak{N}$ is isomorphic to the lattice $2$.
\end{enumerate}
\end{theorem}

\begin{center} 
\large{Acknowledgements}
\end{center}
The authors were supported by ERC-Starting Grant No. 277728 "Geometry and Analysis of Group Rings". They are indebted to the anonymous referee for helpful comments and to Nikolay Nikolov for a considerably shorter proof of Lemma~\ref{lem:LengthFunctionComparisonNumberThreeInSn}. We want to thank Philip Dowerk for helpful discussions and finding many small mistakes. The second author wants to thank Simon Thomas and Nikolay Nikolov for interesting and inspiring discussions.

\bibliographystyle{plain}
\bibliography{usg}{}

\small

\vspace{0.5cm}

\textit{Abel Stolz, Universit\"at Leipzig, Augustusplatz 10, 04109 Leipzig, Germany \\}
\verb|abel.stolz@math.uni-leipzig.de|\\

\textit{Andreas Thom, Universit\"at Leipzig, Augustusplatz 10, 04109 Leipzig, Germany  \\}
\verb|andreas.thom@math.uni-leipzig.de|\\

\end{document}